\documentclass[10pt]{article}

\usepackage[margin=2.0cm]{geometry} 
\usepackage[utf8]{inputenc} 
\usepackage{authblk} 
\usepackage{hyperref} 
\usepackage{color}
\usepackage{array,multirow}

\usepackage{amssymb}
\usepackage{amsmath}
\usepackage{amsthm}
\usepackage{dutchcal}
\usepackage{graphicx}
\usepackage{subcaption}
\usepackage{algorithm} 
\usepackage{algorithmicx}  
\usepackage{algpseudocode}
\usepackage[capitalise]{cleveref}

\theoremstyle{definition}
\newtheorem{theorem}{Theorem}[section]
\newtheorem{proposition}[theorem]{Proposition}
\newtheorem{lemma}[theorem]{Lemma}
\newtheorem{corollary}[theorem]{Corollary}

\theoremstyle{remark}
\newtheorem{remark}[theorem]{Remark}

\newcommand{\R}{\mathbb{R}}
\newcommand{\ucal}{\mathcal{u}}

\newcommand{\Pcal}{\mathcal{P}}
\newcommand{\Xcal}{\mathcal{X}}
\newcommand{\ur}{\widetilde{u}}
\newcommand{\urm}{\widetilde{u}(\boldsymbol{\mu})}

\newcommand{\Mu}{\boldsymbol{\mu}}

\newcommand{\black}[1]{{\color{black} #1}}

\DeclareMathOperator{\Var}{Var}

\DeclareMathOperator{\rank}{rank}

\begin{document}

\title{Randomized residual-based error estimators for the Proper Generalized Decomposition approximation of parametrized problems}

\author{Kathrin Smetana\footnote{Faculty of Electrical Engineering, Mathematics \& Computer Science, University of Twente, Zilverling,
P.O. Box 217, 7500 AE Enschede, The Netherlands, \texttt{k.smetana@utwente.nl}}, Olivier Zahm\footnote{Univ. Grenoble Alpes, CNRS, Inria, Grenoble INP, LJK, 38000 Grenoble, France, \texttt{olivier.zahm@inria.fr}}}

\maketitle

\begin{abstract}
This paper introduces a novel error estimator for the Proper Generalized Decomposition (PGD) approximation of parametrized equations.
The estimator is intrinsically random: It builds on concentration inequalities of Gaussian maps and an adjoint problem with random right-hand side, which we approximate using the PGD. The effectivity of this randomized error estimator can be arbitrarily close to unity with high probability, allowing the estimation of the error with respect to any user-defined norm as well as the error in some quantity of interest.
The performance of the error estimator is demonstrated and compared with some existing error estimators for the PGD for a parametrized time-harmonic elastodynamics problem and the parametrized equations of linear elasticity with a high-dimensional parameter space.\\

\vspace*{\baselineskip}
\noindent
{\bf Keywords:} A posteriori error estimation, parametrized equations, Proper Generalized Decomposition, Monte-Carlo estimator, concentration phenomenon, goal-oriented error estimation
\end{abstract}

\section{Introduction}\label{sec1}
Parametrized models are ubiquitous in engineering applications, life sciences, hazard event prediction and many other areas. Depending on the application, the parameters may, for instance, represent the permeability of soil, the Young's modulus field, the geometry but also uncertainty in the data. In general, the goal is either to analyze the model response with respect to the parameter or to perform real-time simulations for embedded systems. In either case, the complexity of the model is the computational bottleneck, especially if computing the model response for one parameter value requires the solution to a complex system of partial differential equations (PDEs). Model order reduction methods such as the Reduced Basis (RB) method \cite{HeRoSt16,QuMaNe16,Haa17,RoHuPa08,VePrRoPa03}, the Proper Generalized Decomposition (PGD) \cite{AmMoChKe06,AmMoChKe07,Nouy07}, the hierarchical model reduction approach \cite{VogBab81a,PerErnVen10,OhlSme14,SmeOhl17,BTOS16}, or, more generally, tensor-based methods \cite{Hac12,Nou17} are commonly used to build a fast-to-evaluate approximation to the model response. 

This paper is concerned with the PGD, which has been developed for the approximation of models defined in multidimensional spaces, such as problems in quantum chemistry, financial mathematics or (high-dimensional) parametrized models. For the variational version of the PGD, where an energy functional is minimized, convergence has been proved for the Poisson problem in \cite{LeBLelMad09}. For a review on the PGD method see for instance \cite{ChiAmmCue2010,ChLaCu11,ChKeLe14} and we refer to \cite{ChiCue14} for an overview of how to use PGD in different applications such as augmented learning and process optimization. 

The goal of this paper is to introduce a novel error estimator for the PGD approximation.
Error estimation is of crucial importance, in particular 
to stop the PGD enrichment process as soon as the desired precision is reached.
To that end, we extend the randomized residual-based error estimator introduced in \cite{SmZaPa19} to the PGD. The randomized error estimator features several desirable properties: first, it does not require the estimation of stability constants such as a coercivity or inf-sup constant and the effectivity index is close to unity with prescribed lower and upper bounds at high probability. The effectivity index can thus be selected by the user, balancing computational costs and desired sharpness of the estimator. Moreover, the presented framework yields error estimators with respect to user-defined norms, for instance the $L^2$-norm or the $H^1$-norm; the approach also permits error estimation of linear quantities of interest (QoI). Finally, depending on the desired effectivity the computation of the error estimator is in general only as costly as the computation of the PGD approximation or even significantly less expensive, which makes our error estimator strategy attractive from a computational viewpoint. 
To build this estimator, we first estimate the norm of the error with a Monte Carlo estimator using Gaussian random vectors whose covariance is chosen according to the desired error measure, e.g., user-defined norms or quantity of interest. Then, we introduce a dual problem with random right-hand side the solution of which allows us to rewrite the error estimator in terms of the residual of the original equation. Approximating the random dual problems with the PGD yields a fast-to-evaluate error estimator that has low marginal cost. 

Next, we relate our error estimator to other stopping criteria that have been proposed for the PGD solution. In \cite{AmChHu10} the authors suggest letting an error estimator in some linear QoI steer the enrichment of the PGD approximation. In order to estimate the error they employ an adjoint problem with the linear QoI as a right-hand side and approximate this adjoint problem with the PGD as well. It is important to note that, as observed in \cite{AmChHu10}, this type of error estimator seems to require in general a more accurate solution of the dual than the primal problem. In contrast, the framework we present in this paper often allows using a dual PGD approximation with significantly less terms than the primal approximation. An adaptive strategy for the PGD approximation of the dual problem for the error estimator proposed in \cite{AmChHu10} was suggested in \cite{GNLC13} in the context of computational homogenization. In \cite{Alfetal15} the approach proposed in \cite{AmChHu10} is extended to problems in nonlinear solids by considering a suitable linearization of the problem.

An error estimator for the error between the exact solution of a parabolic PDE and the PGD approximation in a suitable norm is derived in \cite{LadCha11} based on the constitutive relation error and the Prager-Synge theorem. The discretization error is thus also assessed, and the computational costs of the error estimator depend on the number of elements in the underlying spatial mesh. Relying on the same techniques an error estimator for the error between the exact solution of a parabolic PDE and a PGD approximation in a linear QoI was proposed in \cite{LadCha12}, where the authors used the PGD to approximate the adjoint problem. In \cite{Moi13} the hypercircle theory introduced by Prager and Synge is invoked to derive an error between the exact solution of the equations of linear elasticity and the PGD approximation. While the estimators in \cite{LadCha11,LadCha12, Moi13} are used to certify the PGD approximation, in \cite{Chetal17} the authors slightly modify the techniques used in \cite{LadCha11,LadCha12} to present an error estimator that is used to drive the construction of the PGD approximation within an adaptive scheme. To that end, the authors assume that also the equilibrated flux can be written in a variable separation form. In contrast to the error estimator we introduce in this paper, the estimators presented in \cite{LadCha11,LadCha12, Moi13, Chetal17} also take into account the discretization error; for the a posteriori error estimator suggested in this paper this is the subject of a forthcoming work. Moreover, in \cite{LadCha11,LadCha12,Moi13, Chetal17} the authors consider (parametrized) coercive elliptic and parabolic problems; it seems that the techniques just described have not been used yet to address problems that are merely inf-sup stable. 

Error estimators based on local error indicators are proposed to drive mesh adaptation within the construction of the PGD approximation in \cite{Nadetal15}. In \cite{BFNoZa14}, an ideal minimal residual formulation is proposed in order to build optimal PGD approximations. In that work, an error estimator with controlled effectivity is derived from the solution of some adjoint problem.
Finally, a PGD approximation of the adjoint problem is used to estimate a QoI in \cite{KCLP19}. The latter is then employed to constrain the primal approximation, aiming at a (primal) PGD approximation more tailored to the considered QoI \cite{KCLP19}.

The randomized a posteriori error estimator proposed in \cite{SmZaPa19}, which we extend in this paper to the PGD, is inspired by the probabilistic error estimator for the approximation error in the solution of a system of ordinary differential equations introduced in \cite{CaoPet04,HoPeSe07}. Here, the authors invoke the small statistical sample method from \cite{KenLau94}, which estimates the norm of a vector by its inner product with a random vector drawn uniformly at random on the unit sphere. 

We note that randomized methods for error estimation are gaining interest in the reduced order modeling community in particular and in the context of parametrized PDEs generally. For instance in \cite{BalNou18}, random sketching techniques are used within the context of the RB method to speed-up the computation of the dual norm of the residual used as an error indicator. Here, the authors exploit that the residual lies in a low-dimensional subspace, allowing them to invoke arguments based on $\varepsilon$-nets in order to certify the randomized algorithm. In \cite{JaNoPr16} for the RB method a probabilistic a posteriori error bound for linear scalar-valued QoIs is proposed, where randomization is done by assuming that the parameter is a random variable on the parameter set. In \cite{zahm2016interpolation}, an interpolation of the operator inverse is built via a Frobenius-norm projection and computed efficiently using randomized methods. An error estimator is obtained by measuring the norm of residual multiplied by the interpolation of the operator inverse, used here as a preconditioner. Next, in the context of localized model order reduction a reliable and efficient probabilistic a posteriori error estimator for the operator norm of the difference between a finite-dimensional linear operator and its orthogonal projection onto a reduced space is derived in \cite{BuhSme18}. Finally in the recent paper \cite{Eetal19} the authors use concentration inequalities to bound the generalization error in the context of statistical learning for high-dimensional parametrized PDEs in the context of Uncertainty Quantification.

The remainder of this article is organized as follows. After introducing the problem setting and recalling the PGD in Section \ref{sec2}, we introduce an unbiased randomized a posteriori error estimator in Section \ref{sec3}. As this error estimator still depends on the high-dimensional solutions of dual problems, we suggest in Section \ref{sec4} how to construct dual PGD approximations for this setting efficiently and propose an algorithm where the primal and dual PGD approximations are constructed in an intertwined manner. In Section \ref{sec5} we compare the effectivity of the randomized a posteriori error estimator with a stagnation-based error estimator and the dual norm of the residual for a time-harmonic elastodynamics problem and a linear elasticity problem with $20$-dimensional parameter space. We close with some concluding remarks in Section \ref{sec6}.

\section{Problem setting and the Proper Generalized Decomposition}\label{sec2}

We denote by $\Mu=(\Mu_1,\hdots,\Mu_p)$ a vector of $p$ parameters which belongs to a set $\mathcal{P}\subset\R^p$ of admissible parameters. Here, we assume that $\mathcal{P}$ admits a product structure $\mathcal{P}=\Pcal_{1} \times\hdots\times \Pcal_{p}$, where $\Pcal_{i}\subset\R$ is the set of admissible parameters for $\Mu_i$, $i=1,\hdots,p$. For simplicity we assume that $\Pcal_{i}=\{\Mu_i^{(1)},\hdots,\Mu_i^{(\#\Pcal_{i})}\}$ is a finite set of parameters (obtained for instance as an interpolation grid or by a Monte Carlo sample of some probability measure) so that $\mathcal{P}$ is a grid containing $\#\Pcal=\prod_{i=1}^d\#\Pcal_{i}$ points.

Let $D \subset \mathbb{R}^{d}$ be a bounded spatial domain and let $\mathcal{A}(\Mu): \mathcal{X} \rightarrow \mathcal{X}'$ be a parametrized linear differential operator, where $\mathcal{X}$ is a suitable Sobolev space of functions defined on $D$; $\mathcal{X}'$ denotes the dual space of $\mathcal{X}$. We consider the following problem: for any $\Mu \in \mathcal{P}$ and $\mathcal{f}(\Mu ) \in \mathcal{X}'$, find $\mathcal{u}(\Mu) \in \mathcal{X}$ such that 
\begin{equation}\label{eq:PDE}
\mathcal{A}(\Mu) \mathcal{u}(\Mu ) = \mathcal{f}(\Mu ) \qquad \text{in } \mathcal{X}'.
\end{equation}
We ensure well-posedness of problem \eqref{eq:PDE} by requiring that
\begin{equation*}\label{eq:inf-sup}
\beta(\Mu):=\inf_{v \in \mathcal{X}} \sup_{w \in \mathcal{X}} \frac{\langle \mathcal{A}(\Mu ) v, w \rangle_{\mathcal{X}^{'},\mathcal{X}}}{\|v\|_{\mathcal{X}}\|w\|_{\mathcal{X}}} > 0 
\qquad\text{and}\qquad
\gamma (\Mu ):=\sup_{v \in \mathcal{X}} \sup_{w \in \mathcal{X}} \frac{\langle \mathcal{A}(\Mu ) v, w \rangle_{\mathcal{X}^{'},\mathcal{X}}}{\|v\|_{\mathcal{X}}\|w\|_{\mathcal{X}}}  < \infty  ,
\end{equation*}
hold for all $\Mu \in \mathcal{P}$. Here $\langle \cdot , \cdot \rangle_{\mathcal{X}^{'},\mathcal{X}}$ denotes the duality pairing of $\mathcal{X}'$ and $\mathcal{X}$.

Next, we introduce a finite dimensional space $\mathcal{X}_{N} := \text{span}\{\psi_1,\hdots,\psi_N\}\subset \mathcal{X}$ defined as the span of $N\gg 1$ basis functions $\psi_1,\hdots,\psi_N$. For any $\boldsymbol{\mu} \in \mathcal{P}$, we define the high fidelity approximation of \eqref{eq:PDE} as the function $\mathcal{u}_{N}(\boldsymbol{\mu}) \in \mathcal{X}_{N}$ that solves the variational problem
\begin{equation}\label{eq:PDE_discrete}
 \langle \mathcal{A}(\Mu) \mathcal{u}_{N}(\Mu ) , w_{N} \rangle_{\mathcal{X}^{'},\mathcal{X}} = \langle\mathcal{f}(\Mu ) , w_{N}\rangle_{\mathcal{X}^{'},\mathcal{X}} \qquad \text{ for all } w_{N}\in\mathcal{X}_{N}.
\end{equation}
Again, we ensure wellposed-ness of problem \eqref{eq:PDE_discrete} by requiring the discrete inf-sup/sup-sup conditions 
\begin{equation}\label{eq:inf-sup_discrete} 
\beta_{N}(\Mu )  := \inf_{v \in \mathcal{X}_{N}} \sup_{w \in \mathcal{X}_{N}} \frac{\langle \mathcal{A}(\Mu ) v, w \rangle_{\mathcal{X}^{'},\mathcal{X}}}{\|v\|_{\mathcal{X}}\|w\|_{\mathcal{X}}} > 0 
\qquad\text{and}\qquad
\gamma_{N} (\Mu ):=\sup_{v \in \mathcal{X}_{N}} \sup_{w \in \mathcal{X}_{N}} \frac{\langle \mathcal{A}(\Mu ) v, w \rangle_{\mathcal{X}^{'},\mathcal{X}}}{\|v\|_{\mathcal{X}}\|w\|_{\mathcal{X}}}  < \infty ,
\end{equation}
to hold for all $\Mu \in \mathcal{P}$. Expressing $\mathcal{u}_{N}(\Mu )$ in the basis $\mathcal{u}_{N}(\Mu ) = \sum_{i=1}^{N} u_{i}(\Mu ) \psi_{i}$ and defining the vector of coefficient $u(\Mu ):=(u_{1}(\Mu ),\hdots,u_{N}(\Mu )) \in \R^{N}$, we rewrite \eqref{eq:PDE_discrete} as a parametrized algebraic system of linear equations
\begin{equation}\label{eq:AUB}
A(\Mu) u(\Mu ) = f(\Mu ),
\end{equation}
where the matrix $A(\Mu ) \in \R^{N\times N}$ and the vector $f(\Mu ) \in \R^{N}$ are defined as $A_{ij}(\Mu ) = \langle \mathcal{A}(\Mu )\psi_{j},\psi_{i} \rangle_{\mathcal{X}^{'},\mathcal{X}}$ and $f_i(\Mu ) = \langle \mathcal{f}(\Mu ),\psi_{i} \rangle_{\mathcal{X}^{'},\mathcal{X}}$ for any $1\leq i,j\leq N$. Finally, with a slight abuse of notation, we denote by $\|\cdot\|_{\mathcal{X}_N}$ and $\|\cdot\|_{\mathcal{X}_N'}$ the norms defined on the algebraic space $\R^N$ such that 
\begin{equation}\label{eq:NormsDefinition}
 \|v\|_{\mathcal{X}_N} := \| \sum_{i=1}^N v_i \psi_i \|_\mathcal{X},
 \qquad\text{and}\qquad
 \|v\|_{\mathcal{X}_N'} := \sup_{w \in \mathcal{X}_{N}} \frac{\langle \sum_{i=1}^N v_i \psi_i, w \rangle_{\mathcal{X}^{'},\mathcal{X}}}{\|w\|_{\mathcal{X}}} \qquad \text{ for any } v\in\R^N.
\end{equation}
These norms can be equivalently defined by $\|v\|_{\mathcal{X}_N}^2 = v^TR_{\Xcal}v$ and $\|v\|_{\mathcal{X}_N'}^2 = v^TR_{\Xcal}^{-1}v$ for any $v\in\R^N$, where $R_{\Xcal}\in\R^{N\times N}$ is the gram matrix of the basis $(\psi_1,\hdots,\psi_N)$ such that $(R_{\Xcal})_{ij}=(\psi_i,\psi_j)_\mathcal{X}$ where $(\cdot,\cdot)_\mathcal{X}$ is the scalar product in $\Xcal$.

\subsection{The Proper Generalized Decomposition}\label{subsec:PGD}

To avoid the curse of dimensionality we compute an approximation of $u(\Mu)=u(\Mu_1,\hdots,\Mu_p) \in\R^N$ using the Proper Generalized Decomposition (PGD). The PGD relies on the separation of the variables $\Mu_1,\hdots,\Mu_p$ to build an approximation of the form
\begin{equation}\label{eq:tensor_appr}
\widetilde{u}^{M}(\Mu_{1},\hdots ,\Mu_{p}) = \sum_{m=1}^{M} u^{m}_{\mathcal{X}} ~  u^{m}_{1}(\Mu_1)  \cdots  u^{m}_{p}(\Mu_p), 
\end{equation} 
where the $M$ vectors $u^{m}_{\mathcal{X}} \in\R^N$ and the $p\times M$ scalar univariate functions $u^{m}_{i}:\Pcal_i\rightarrow\R$ need to be constructed. 
An approximation as in \eqref{eq:tensor_appr} can be interpreted as a low-rank tensor approximation in the canonical format, see \cite{Nou17}. With this interpretation, $M$ is called the canonical rank of $\widetilde{u}^{M}$ seen as a tensor $\widetilde{u}^{M} = \sum_{m=1}^{M} u^{m}_{\mathcal{X}} \otimes u^{m}_{1} \otimes\hdots\otimes  u^{m}_{p}$. 
Let us note that if the space of spatial functions $\mathcal{X}$ was a tensor product space $\mathcal{X}=\mathcal{X}^1\otimes\hdots\otimes \mathcal{X}^d$, then one could have further separated the spatial variables and obtained approximations like $\widetilde{u}^{M} = \sum_{m=1}^{M} u^{m}_{\mathcal{X}^1} \otimes\hdots\otimes  u^{m}_{\mathcal{X}^d} \otimes u^{m}_{1} \otimes\hdots\otimes  u^{m}_{p}$. This idea is however not pursued in this paper.
The PGD is fundamentally an $\ell^2$-based model order reduction technique in the sense that it aims at constructing $\widetilde{u}^{M}$ such that the root mean squared (RMS) error
\begin{equation}\label{eq:IdealPGDError}
 \|u-\widetilde{u}^M \| := \sqrt{\frac{1}{\#\mathcal{P}}\sum_{\Mu\in\mathcal{P}} \|u(\Mu)-\widetilde{u}^M(\Mu)\|_{\mathcal{X}_N}^2 },
\end{equation}
is controlled. However, minimizing this error over the functions $\widetilde{u}^{M}$ as in \eqref{eq:tensor_appr} is not feasible in practice. We now recall the basic ingredients of the PGD for the computation of $\widetilde{u}^M(\Mu)$, referring for instance to \cite{ChiAmmCue2010,ChLaCu11,ChKeLe14} for further details.

\subsubsection*{Progressive PGD}

The progressive PGD relies on a greedy procedure where we seek the $M$ terms in \eqref{eq:tensor_appr} one after another. After $M-1$ iterations, we build a rank-one correction $u^{M}_{\mathcal{X}} ~  u^{M}_{1}(\Mu_1)  \hdots  u^{M}_{p}(\Mu_p)$ of the current approximation $\widetilde{u}^{M-1}(\Mu)$ by minimizing some error functional described below. Once this correction is computed, we update the approximation 
\begin{equation}\label{eq:GreedyPGD}
 \widetilde{u}^{M}(\Mu) = \widetilde{u}^{M-1}(\Mu) + u^{M}_{\mathcal{X}} ~  u^{M}_{1}(\Mu_1)  \cdots  u^{M}_{p}(\Mu_p),
\end{equation}
and we increment $M\leftarrow M+1$. Such a procedure is called \emph{pure greedy} because it never updates the previous iterates.

We now describe how to construct the rank-one correction. A commonly used error functional is the mean squared residual norm, which yields the so-called \emph{Minimal Residual PGD} \cite{BFNoZa14,FHMM13,CaEhLe13}. The correction $u^{M}_{\mathcal{X}} ~  u^{M}_{1}(\Mu_1)  \cdots  u^{M}_{p}(\Mu_p)$ is defined as a solution to
\begin{equation}\label{eq:MinResPGD}
\min_{u_{\mathcal{X}}\in\R^N} 
 \min_{u_{1}:\mathcal{P}_1\rightarrow\R} 
 \cdots
 \min_{u_{p}:\mathcal{P}_p\rightarrow\R} 
 ~ \sum_{\Mu\in\mathcal{P}} \|A(\Mu)\Big(\widetilde{u}^{M-1}(\Mu) + u_{\mathcal{X}} u_{1}(\Mu_1)  \cdots  u_{p}(\Mu_p)\Big) - f(\Mu)\|_{2}^2,
\end{equation}
where $\|\cdot\|_2$ denotes the canonical norm of $\R^N$. When the matrix $A(\Mu)$ is symmetric definite positive (SPD), we can alternatively use the \emph{Galerkin PGD} which consists in minimizing the mean squared energy
\begin{equation}\label{eq:MinEnergyPGD}
\min_{u_{\mathcal{X}}\in\R^N} 
 \min_{u_{1}:\mathcal{P}_1\rightarrow\R} 
 \cdots
 \min_{u_{p}:\mathcal{P}_p\rightarrow\R} 
 ~ \sum_{\Mu\in\mathcal{P}} \|A(\Mu)\Big(\widetilde{u}^{M-1}(\Mu) + u_{\mathcal{X}} u_{1}(\Mu_1)  \cdots  u_{p}(\Mu_p)\Big) - f(\Mu)\|_{A(\Mu)^{-1}}^2 ,
\end{equation}
where the residual is measured with the norm $\|\cdot\|_{A(\Mu)^{-1}} = \sqrt{(\cdot)^TA(\Mu)^{-1}(\cdot)}$. The Galerkin PGD is usually much more efficient compared to the Minimal Residual PGD, but the latter is much more stable for non SPD operators. Both for the Minimal Residual PGD \eqref{eq:MinResPGD} and the Galerkin PGD \eqref{eq:MinEnergyPGD} Conditions \eqref{eq:inf-sup_discrete} are sufficient to ensure the convergence of the pure greedy algorithm outlined above, see \cite{BFNoZa14}. 

A simple and yet efficient approach to solve \eqref{eq:MinResPGD} or \eqref{eq:MinEnergyPGD} is the \emph{alternating least squares} (ALS) method. It relies on the fact that each of the $p+1$ minimization problems is a least squares problem (the cost function being quadratic with respect to each of the $u_\mathcal{X},u_1,u_2,\hdots$) and thus is simple to solve. ALS consists in iteratively solving each of these least squares problems one after the other while keeping the other variables fixed. That is: starting with a random initialization $u_{\mathcal{X}}^{(0)}, u_{1}^{(0)},  \hdots,  u_{p}^{(0)}$, we first compute $u_{\mathcal{X}}^{(1)}$ solution to \eqref{eq:MinResPGD} (or to \eqref{eq:MinEnergyPGD}) with $u_{1}=u_{1}^{(0)},  \hdots,  u_{p}=u_{p}^{(0)}$ fixed. Then we compute $u_{1}^{(1)}$ solution to \eqref{eq:MinResPGD} (or to \eqref{eq:MinEnergyPGD}) with $u_{\mathcal{X}}=u_{\mathcal{X}}^{(1)}, u_{2}=u_{2}^{(0)}, \hdots,  u_{p}=u_{p}^{(0)}$ fixed etc. This yields a sequence of updates $u_{\mathcal{X}}^{(1)}\rightarrow u_{1}^{(1)}\rightarrow\cdots\rightarrow u_{p}^{(1)}\rightarrow u_{\mathcal{X}}^{(2)}\rightarrow u_{1}^{(2)} \rightarrow\cdots$ which we stop once the stagnation is reached. In practice, few iterations (three or four) suffice.

\subsubsection*{Stopping criterion for the greedy enrichment process}

Ideally we want to stop the greedy iteration process in $M$ as soon as the relative RMS error $\|u-\widetilde{u}^{M}\|/\|u\|$ falls below a certain prescribed tolerance, which is why an error estimator for the relative RMS error is desirable. In \emph{stagnation-based error estimators} the unknown solution $u$ is replaced by a reference solution $\widetilde{u}^{M+k}$ for some $k\geq1$
\begin{equation}\label{eq:StagnationErrorEstimator}
 \Delta_{M,k}^\text{stag} := \frac{\|\widetilde{u}^{M+k}-\widetilde{u}^{M}\|}{\|\widetilde{u}^{M+k}\|},
\end{equation}
where the norm $\|\cdot\|$ is defined as in \eqref{eq:IdealPGDError}.
The estimator \eqref{eq:StagnationErrorEstimator} is also often called a \emph{hierarchical error estimator}, considering two different PGD approximations: a coarse one $\widetilde{u}^{M}$ and a fine one $\widetilde{u}^{M+k}$. The quality of such an error estimator depends on the convergence rate of the PGD algorithm: a slow PGD convergence would require choosing a larger $k$ to obtain an accurate estimate of the relative RMS error. Needless to say that the error estimator $\Delta_{M,k}^\text{stag}$ is more expensive to compute when $k\gg1$. Proving that $\Delta_{M,k}^\text{stag}$ is an upper bound of the error and bounding the effectivity typically requires satisfying a so-called saturation assumption, i.e., having a uniform bound on the convergence rate of the PGD, a requirement which is hard to ensure in practice.

As one example of a \emph{residual-based error estimator} we consider the relative RMS residual norm 
\begin{equation}\label{eq:RelativeResNorm}
 \Delta_{M}^\text{res} := \sqrt{\frac{\frac{1}{\#\mathcal{P}}\sum_{\Mu\in\mathcal{P}} \|A(\Mu)\widetilde{u}^M(\Mu)-f(\Mu)\|_{\mathcal{X}_N'}^2}{\frac{1}{\#\mathcal{P}}\sum_{\Mu\in\mathcal{P}} \|f(\Mu)\|_{\mathcal{X}_N'}^2 }},
\end{equation}
where the norm $\|\cdot\|_{\mathcal{X}_N'}$ has been defined in \eqref{eq:NormsDefinition}. 
The advantage of this error estimator is that it comes with the guarantee that
\begin{equation}\label{eq:RelativeResNormEffectivity}
 \kappa_N^{-1} \frac{\|u-\widetilde{u}^{M}\|}{\|u\|} \leq  \Delta_{M}^\text{res} \leq \kappa_N \frac{\|u-\widetilde{u}^{M}\|}{\|u\|} ,
\end{equation}
for any $M$, where $\kappa_N = \max_{\Mu\in\mathcal{P}} \gamma_N(\Mu)/\beta_N(\Mu) \geq1$ denotes the supremum of the condition number of $A(\Mu)$ seen as an operator from $(\R^N,\|\cdot\|_{\mathcal{X}_N})$ to $(\R^N,\|\cdot\|_{\mathcal{X}_N'})$. This inequality can be easily proved using Assumption \eqref{eq:inf-sup_discrete} and Definition \eqref{eq:NormsDefinition}. Inequality \eqref{eq:RelativeResNormEffectivity} ensures that the effectivity of $\Delta_{M}^\text{res}$ is contained in $[\kappa_N^{-1},\kappa_N]$. Note however that estimating $\kappa_N$ (or a sharp upper bound of it) is difficult in practice as it requires estimating the inf-sup/sup-sup constants $\beta_N(\Mu)$ and $\gamma_N(\Mu)$, a task that is general rather costly \cite{Hetal07,Cetal09,Hetal10}. Moreover, for ill-conditioned problems as, say, the Helmholtz equation or a time-harmonic elastodynamics problem presented in section \ref{sec5}, the constant $\beta_N(\Mu)$ can be arbitrarily close to zero, yielding a very large $\kappa_N$.

In \cite{AmChHu10} the authors suggest letting an estimator for the error in some linear QoI steer the enrichment of the PGD approximation. In order to estimate the error they employ an adjoint problem with the linear QoI as a right-hand side and approximate this adjoint problem with the PGD as well. In \cite{LadCha11} the constitutive relation error and the Prager-Synge theorem is employed to derive an error estimator for the error between the exact solution of the PDE and the PGD approximation in a suitable norm. The error estimator requires the construction of equilibrated fluxes and has been proposed to use within an adaptive scheme to construct the PGD approximation in \cite{Chetal17}.

\section{A randomized a posteriori error estimator}\label{sec3}

In this paper we aim at estimating the errors
\begin{equation}\label{eq:error}
\|u(\Mu)-\widetilde{u}^M(\Mu)\|_{\Sigma},
\qquad
 \sqrt{\frac{1}{\#\mathcal{P}}\sum_{\Mu\in\mathcal{P}} \|u(\Mu)-\widetilde{u}^M(\Mu)\|_{\Sigma}^2},
 \qquad\text{or}\qquad
 \sqrt{\frac{\frac{1}{\#\mathcal{P}}\sum_{\Mu\in\mathcal{P}} \|u(\Mu)-\widetilde{u}^M(\Mu)\|_{\Sigma}^2}{\frac{1}{\#\mathcal{P}}\sum_{\Mu\in\mathcal{P}} \|u(\Mu)\|_{\Sigma}^2}},
\end{equation}
where $\Sigma\in\mathbb{R}^{N\times N}$ is a symmetric positive semi-definite matrix which is chosen by the user according to the desired error measure, see subsection \ref{subsec:3.1}. This means that $\|\cdot\|_\Sigma = \sqrt{(\cdot)^T\Sigma(\cdot)}$ is a semi-norm unless $\Sigma$ is SPD, in which case $\|\cdot\|_\Sigma$ is a norm.
We show in subsection \ref{subsec:estimateNorm} that, at high probability, we can estimate accurately $\|v\|_\Sigma$ for any $v\in\R^N$ by means of \emph{random Gaussian maps}. In subsection \ref{subsec:estimate many vectors} we use union bound arguments to estimate the errors $\|u(\Mu)-\widetilde{u}^M(\Mu)\|_{\Sigma}$ and the norms $\|u(\Mu)\|_{\Sigma}$ simultaneously for a finite number of parameters $\Mu$,  which allows to estimate all quantities in \eqref{eq:error}.

\subsection{Choice of the error measure}\label{subsec:3.1}

If we wish to estimate the error in the natural norm $\|\cdot\|_{\Sigma}=\|\cdot\|_{\mathcal{X}_N}$ defined in \eqref{eq:NormsDefinition}, we simply choose $\Sigma=R_{\Xcal}$, where we recall that $(R_{\Xcal})_{ij}=(\psi_i,\psi_j)_\mathcal{X}$ is the gram matrix of the basis $(\psi_1,\hdots,\psi_N)$. 
Alternatively, we can consider the error in the $L^{2}(D)$-norm by letting $\Sigma=R_{L^{2}(D)}$ where the matrix $R_{L^{2}(D)}$ is defined by $(R_{L^{2}(D)})_{ij}=(\psi_i,\psi_j)_{L^{2}(D)}$. The presented framework also encompasses estimating the error in some vector-valued quantity of interest defined as a linear function of $u(\Mu)$, say 
$$
 s(\Mu)=L  u(\Mu) ~ \in\mathbb{R}^m ,
$$
for some extractor $L\in\mathbb{R}^{m \times N}$.
In this context we are interested in estimating $\| s(\Mu) - L\,\widetilde u(\Mu) \|_{\mathcal{W}}$ where $\|\cdot\|_{\mathcal{W}}$ is some norm associated with a SPD matrix $R_{\mathcal{W}}\in\R^{m\times m}$. For example, one can be interested in computing the approximation error over a subdomain $D_{sub} \subsetneq D$, that is, estimating the error $\ucal (\Mu )|_{D_{sub}}  - \widetilde \ucal (\Mu )|_{D_{sub}}$ measured in the $L^2(D_{sub})$-norm. In this situation, $L$ would extract the relevant entries from the vector $\urm$ and $R_{\mathcal{W}}$ would be defined by  $(R_{\mathcal{W}})_{ij} = (\psi_i|_{D_{sub}},\psi_j|_{D_{sub}})_{L^{2}(D_{sub})}$. 
With the choice $\Sigma = L^T R_W L$ we can write
\begin{align*}
 \| u(\Mu) - \widetilde u(\Mu) \|_\Sigma^2
 &= (u(\Mu)-\widetilde u(\Mu))^T \big( L^T R_W L \big) (u(\Mu)-\widetilde u(\Mu)) =  \| s(\Mu) - L\widetilde u(\Mu) \|_{\mathcal{W}}^2 ,
\end{align*}
so that measuring the error with respect to the norm $\|\cdot\|_\Sigma$ gives the error associated with the QoI. Notice that if $m<N$ the matrix $\Sigma$ is necessarily singular and then $\|\cdot\|_\Sigma$ is a semi-norm. Finally, consider the scalar-valued QoI given by $s(\Mu)= l^T u(\Mu)$ where $l\in\mathbb{R}^N$. This corresponds to the previous situation with $m=1$ and $L=l^T$. The choice $\Sigma = l\,l^T $ yields $\| u(\Mu) - \widetilde u(\Mu) \|_\Sigma^2= |s(\Mu) - L\widetilde u(\Mu) |$, where $|\cdot|$ denotes the absolute value.

\subsection{Estimating norms using Gaussian maps}\label{subsec:estimateNorm}

Random maps have proved themselves to be a very powerful tool to speed up computations in numerical linear algebra. They can also be exploited to estimate the norm of a vector very efficiently. In this subsection, we will present the required basics to first obtain an estimator for the norm of a fixed vector and then for a set of vectors. The latter is for instance of relevance if we want to compute the error \eqref{eq:error}. For a thorough introduction on how to estimate certain quantities in numerical linear algebra by random maps we refer to the recent monograph \cite{Ver18}; we follow subsection 2.2 of \cite{SmZaPa19} here.

First, we will define a suitable random map $\Phi: \mathbb{R}^{N}  \rightarrow \R^{K}$ with $K\ll N$ that can be used to estimate the (semi-)norm $\|\cdot\|_\Sigma$ by $\|\Phi\cdot\|_2$, where $\|\cdot\|_2$ is the Euclidean norm in $\mathbb{R}^K$. To that end, let $Z\sim\mathcal{N}(0,\Sigma)$ be a zero mean Gaussian random vector in $\mathbb{R}^N$ with covariance matrix $\Sigma\in\mathbb{R}^{N\times N}$. We emphasize that the user first selects the desired error measure which yields the (semi-)norm $\|\cdot\|_\Sigma$ and, subsequently, the covariance matrix is chosen as that very same matrix. Given a vector $v\in\mathbb{R}^N$ we can write
$$
 \|v\|_\Sigma^2 
 = v^T \Sigma v 
 = v^T \mathbb{E}( ZZ^T )v 
 = \mathbb{E}( (Z^T v)^2 ),
$$
where $\mathbb{E}(\cdot)$ denotes the expected value. This means that $(Z^T v)^2$ is an unbiased estimator of $\|v\|_\Sigma^2$. We may then use $K$ independent copies of $Z$, which we denote by $Z_1,\hdots,Z_K$, to define the random map $\Phi\in\mathbb{R}^{K\times N}$ as $\Phi_{(i,:)}=(1/\sqrt{K})Z_i^T$, where $(i,:)$ denotes the $i$-th row. Then, we can write
\begin{equation}\label{eq:defPhi}
 \| \Phi v \|_2^2 = \frac{1}{K}\sum_{i=1}^K  ( Z_i^T v )^2 \qquad \text{for any } v\in \mathbb{R}^N.
\end{equation}
We see that $\| \Phi v \|_2^2$ is a $K$-sample Monte-Carlo estimator of $\mathbb{E}( (Z^T v )^2 ) = \|v\|_\Sigma^2$. By the independence of the $Z_i$'s, the variance of $\| \Phi v \|_2^2$ is $\Var( \| \Phi v \|_2^2  ) = \frac{1}{K} \Var( (Z^Tv)^2  )$. Increasing $K$ thus allows to reduce the variance of the estimate, and we get a more precise approximation of $\|v\|_\Sigma^2$. However, the variance is not always the best criterion to access the quality of the estimator.

Given a vector $v$ and a realization of $\Phi$, how much $\| \Phi v \|_2^2$ underestimates or overestimates $\|v\|_\Sigma^2$? Because $\Phi$ is a random variable, the answer to that question will hold only up to a certain probability which we need to quantify. To that end, we first note that the random variables $(Z_i^T v ) / \|v\|_\Sigma$ for $i=1,\hdots,K$ are independent standard normal random variables so that we have
\begin{equation}\label{eq:chi-squared}
 \| \Phi v \|_2^2 = \frac{\|v\|_\Sigma^2}{K} \sum_{i=1}^K \Big( \frac{Z_i^T v}{\|v\|_\Sigma}\Big)^2 \sim \frac{\|v\|_\Sigma^2}{K} Q ,
\end{equation}
where $Q \sim \chi^2(K)$ is a random variable which follows a chi-squared distribution with $K$ degrees of freedom. Denoting by $\mathbb{P}\{A\}$ the probability of an event $A$ and by $\overline{A}$ the complementary event of $A$, the previous relation yields
\begin{align*}
 \mathbb{P}\Big\{ w^{-1}\|v\|_\Sigma \leq \| \Phi v \|_2 \leq w \|v\|_\Sigma \Big\}
 &=1-\mathbb{P}\big\{  \overline{ Kw^{-2} \leq  Q \leq Kw^2 }  \big\} ,
\end{align*}
for any $w\geq1$. Then for any given (fixed) vector $v\in \mathbb{R}^N$, the probability that a realization of $\| \Phi v \|_2$ lies between $w^{-1}\|v\|_\Sigma$ and $w\|v\|_\Sigma$ is independent of $v$ but also independent of the dimension $N$. The following proposition gives an upper bound for the failure probability $\mathbb{P}\big\{  \overline{ Kw^{-2} \leq  Q \leq Kw^2 }  \big\} $ in terms of $w$ and $K$.

\begin{proposition}[Proposition 2.1 in \cite{SmZaPa19}]\label{prop:Chi2Tail}
 Let $Q\sim\chi^2(K)$ be a chi-squared random variable with $K\geq3$ degrees of freedom. For any $w>\sqrt{e}$ we have 
 \begin{equation*}\label{eq:Chi2Tail}
  \mathbb{P}\big\{  \overline{ Kw^{-2} \leq  Q \leq Kw^2 }  \big\} 
  \leq \Big( \frac{\sqrt{e}}{w} \Big)^{K} .
 \end{equation*}
\end{proposition}
\begin{proof}
The proof relies on the fact that we have closed form expressions for the law of $Q\sim\chi^2(K)$; for details see \cite{SmZaPa19}.
\end{proof}

Proposition \ref{prop:Chi2Tail} shows that the probability $\mathbb{P}\big\{  \overline{ Kw^{-2} \leq  Q \leq Kw^2 }  \big\}$ decays at least exponentially with respect to $K$, provided $w\geq\sqrt{e}$ and $K\geq3$. 
Then for any $v\in \mathbb{R}^N$, the relation
\begin{equation}\label{eq:controlError}
 w^{-1}\|v\|_\Sigma \leq \| \Phi v \|_2 \leq w \|v\|_\Sigma ,
\end{equation}
holds with a probability greater than $1-(\sqrt{e}/w)^K$. We highlight that this is a non-asymptotic result in the sense that it holds for finite $K$, not relying on the law of large numbers; for further details on those types of non-asymptotic concentration inequalities we refer to \cite{Ver18}.
Moreover, note that by accepting a larger $w$ and thus a larger deviation of the estimator from $\| v \|_{\Sigma}$ fewer samples are required to ensure that the probability of failure $(\sqrt{e}/w)^K$ is small. For instance with $w=2$ and $K=48$, relation \eqref{eq:controlError} holds with a probability larger than $0.9998$, while for $w=4$ we only need  $K=11$ samples to obtain the same probability (see Table \ref{tab:concentration_set}).

Next, by combining Proposition \ref{prop:Chi2Tail} with a union bound argument we can quantify the probability that relation \eqref{eq:controlError} holds simultaneously for any vector in a finite set $\mathcal{M}\subset \mathbb{R}^N$.

\begin{corollary}\label{prop:estimate many vectors}
Given a finite collection of vectors $\mathcal{M}=\{v_1,v_2,\hdots, v_{\# \mathcal{M}}\}\subset \black{\mathbb{R}^N}$ and a failure probability $0<\delta<1$. Then, for any $w > \sqrt{e}$ and 
\begin{equation}\label{eq:ConditionK_FiniteSet}
  K\geq \min \left\lbrace \frac{\log( \# \mathcal{M}) + \log(\delta^{-1})}{\log(w/\sqrt{e})} , \enspace 3 \right \rbrace
\end{equation}
we have
\begin{equation}\label{eq:estimate many vectors}
 \mathbb{P}\Big\{ w^{-1}\|v\|_\Sigma \leq \| \Phi v \|_2  \leq w  \|v\|_\Sigma ~,~\forall v\in\mathcal{M} \Big\} \geq 1 - \delta. 
\end{equation}
\end{corollary}

Table \ref{tab:concentration_set} gives numerical values of $K$ that satisfy \eqref{eq:ConditionK_FiniteSet} depending on $\delta$, $w$ and $\#\mathcal{M}$. For example with $\delta=10^{-4}$ and $w=10$, estimating simultaneously the norm of $10^{9}$ vectors requires only $K=17$ samples. Again, we highlight that this result is independent on the dimension $N$ of the vectors to be estimated. Moreover, Condition \eqref{eq:ConditionK_FiniteSet} allows reducing the number of samples $K$ by increasing $w$\footnote{Note that this is in contrast to the celebrated Johnson-Lindenstrauss lemma \cite{dasgupta1999elementary,johnson1984extensions}.The latter ensures that for any $0<\varepsilon<1$ and any finite set $\mathcal{M}\subset\mathbb{R}^N$ that contains the zero vector, the condition $K\geq 8\varepsilon^{-2}\log(\#\mathcal{M})$ ensures the existence of a linear map $\Phi: \mathbb{R}^{N} \rightarrow \mathbb{R}^{K}$  such that
$ \sqrt{1-\varepsilon} \|v\|_2 \leq \| \Phi v \|_2 \leq \sqrt{1+\varepsilon} \| v \|_2 $ for all $v\in\mathcal{M}$. Increasing $\varepsilon$ does therefore not allow reducing the number of samples $K$; for a more detailed discussion see \cite{SmZaPa19}.}. Note that the computational complexity of the estimator $\| \Phi v \|_2$ and thus the a posteriori error estimator we propose in this paper depends on $K$ (see section \ref{sec4}). By relying on Corollary \ref{prop:estimate many vectors} the user may thus choose between a less accurate but also less expensive estimator (higher $w$, lower $K$) or a more accurate but also more expensive estimator (lower $w$, higher $K$).

\begin{table}[t]
  \centering
  \footnotesize
  
  \begin{tabular}{| l | c c c | c | l | c c c |} 
      \cline{2-4}\cline{7-9}
      \multicolumn{1}{c|}{$\delta=10^{-2}$}& $w=2$ & $w=4$ & $w=10$ &
      \multicolumn{1}{c}{$\quad$}& 
      \multicolumn{1}{c|}{$\delta=10^{-4}$}& $w=2$ & $w=4$ & $w=10$ 
      \\  
      \cline{1-4}\cline{6-9}
      $\#\mathcal{M}=10^0$   & 24 & 6  &  3  &     & $\#\mathcal{M}=10^0$ & 48  & 11 & 6  \\ 
      $\#\mathcal{M}=10^3$   & 60 & 13 &  7  &     & $\#\mathcal{M}=10^3$ & 84  & 19 & 9  \\ 
      $\#\mathcal{M}=10^6$   & 96 & 21 & 11  &     & $\#\mathcal{M}=10^6$ & 120 & 26 & 13  \\ 
      $\#\mathcal{M}=10^9$   & 132& 29 & 15  &     & $\#\mathcal{M}=10^9$ & 155 & 34 & 17  \\ 
      \cline{1-4}\cline{6-9}
  \end{tabular} 
  
 \caption{Minimal value of $K$ for which Condition \ref{eq:ConditionK_FiniteSet} is satisfied.
}
\label{tab:concentration_set}
\vspace{-15pt}
\end{table}

\begin{remark}[Drawing Gaussian vectors]
In actual practice we can draw efficiently from $Z\sim\mathcal{N}(0,\Sigma)$ using \emph{any} factorization of the covariance matrix of the form of $\Sigma=U^{T}U$, for instance a Cholesky decomposition or a symmetric matrix square root.
It is then sufficient to draw a standard Gaussian vector $\widehat{Z}$ and to compute the matrix-vector product $Z=U^{T}\widehat{Z}$. As pointed-out in \cite[Remark 2.9]{BalNou18}, one can take advantage of a potential block structure of $\Sigma$ to build a (non-square) factorization $U$ with a negligible computational cost.
\end{remark}

\subsection{Randomized a posteriori error estimator}\label{subsec:estimate many vectors}

We apply the methodology described in the previous subsection to derive residual-based randomized a posteriori error estimators. Recall the definition of the random map $\Phi: \R^{N} \rightarrow \R^{K}$, $\Phi_{(i,:)}=(1/\sqrt{K})Z_i^T$, $i=1,\hdots,K$, where $(i,:)$ denotes the $i$-th row and $Z_1,\hdots Z_K$ are independent copies of $Z\sim\mathcal{N}(0,\Sigma)$. We may then define the estimators
\begin{equation}\label{defeq:error est}
 \Delta(\Mu)  
 = \| \Phi \big( u(\Mu)-\widetilde u(\Mu) \big) \|_2 
 = \sqrt{ \frac{1}{K} \sum_{k=1}^{K} \Big(Z_{i}^{T} \big(u(\Mu)-\widetilde u(\Mu) \big) \Big)^{2} } 
 \qquad\text{and}\qquad
 \Delta = \sqrt{ \frac{1}{\#\mathcal{P}} \sum_{\Mu \in\mathcal{P}}  \Delta(\Mu)^{2} }
\end{equation}
of the errors $\|u(\Mu)-\widetilde u(\Mu)\|_\Sigma $ and $\sqrt{\frac{1}{\#\mathcal{P}}\sum_{\Mu\in\mathcal{P}}\| u(\Mu)-\widetilde u(\Mu)\|_\Sigma^2}$, respectively.
As both $\Delta(\Mu)$ and $\Delta$ require the high-dimensional solution $u(\Mu)$ of problem \eqref{eq:AUB}, these estimators are computationally infeasible in practice. By introducing the residual
\begin{equation}\label{eq:PrimalResidual}
 r(\Mu)= f(\Mu) - A(\Mu)\widetilde u(\Mu),
\end{equation}
associated with Problem \ref{eq:AUB} and by exploiting the \emph{error residual relationship} we rewrite the terms $Z_{i}^{T} (u(\Mu)- \widetilde u(\Mu))$ in \eqref{defeq:error est} as 
\begin{equation}\label{eq:err res relationship}
 Z_{i}^{T} (u(\Mu) - \widetilde u(\Mu)) = Z_{i}^{T} A(\Mu)^{-1} r(\Mu) = (A(\Mu)^{-T}Z_{i})^{T}r(\Mu) = Y_i(\Mu)^{T}r(\Mu),
\end{equation}
where $Y_{i}(\Mu) \in \mathbb{R}^N$ are the solutions to the \emph{random dual problems} 
\begin{equation}\label{eq:randomDualProblem}
 A(\Mu)^T Y_i(\Mu) = Z_i, \quad 1\leq i \leq K.
\end{equation}
Because the right hand side in \eqref{eq:randomDualProblem} is random, the solutions $Y_1(\Mu),\hdots,Y_K(\Mu)$ are random vectors. 
Using relation \eqref{eq:err res relationship} the error estimators $\Delta(\Mu)$ and $\Delta$ \eqref{defeq:error est} can be rewritten as
\begin{equation}
 \Delta(\Mu)  = \sqrt{ \frac{1}{K} \sum_{k=1}^{K} \Big( Y_i(\Mu)^T r(\Mu) \Big)^2}
 \qquad\text{and}\qquad
 \Delta = \sqrt{ \frac{1}{\#\mathcal{P}} \sum_{\Mu \in\mathcal{P}}  \Delta(\Mu)^{2} }.
 \label{eq:DualTrick}
\end{equation}
This shows that $\Delta(\Mu)$ can be computed by applying $K$ linear forms to the residual $r(\Mu)$.
In that sense, $\Delta(\Mu)$ can be interpreted as an a posteriori error estimator. However, computing the solutions to \eqref{eq:randomDualProblem} is in general as expensive as solving the primal problem \eqref{eq:AUB}. In the next section, we show how to approximate the dual solutions $Y_1(\Mu),\hdots,Y_K(\Mu)$ via the PGD in order to obtain a fast-to-evaluate a posteriori error estimator.
Before that, a direct application of Corollary \ref{prop:estimate many vectors} with $\mathcal{M}=\{u(\Mu)-\widetilde u(\Mu) : \Mu\in\mathcal{P}\}$ gives the following control of the quality of $\Delta(\Mu)$ and $\Delta$.
\begin{corollary}\label{coro:truth est S}
 Given $0 < \delta < 1$ and $w > \sqrt{e}$, the condition
 \begin{equation}\label{coro:truth est Condition S}
   K\geq \min \left\lbrace \frac{\log( \# \mathcal{P}) + \log(\delta^{-1})}{\log(w/\sqrt{e})} , \enspace 3 \right \rbrace,
 \end{equation}
 is sufficient to ensure
 $$
  \mathbb{P}\Big\{ w^{-1} \Delta(\Mu) \leq \|u(\Mu)-\widetilde u(\Mu) \|_\Sigma \leq w \Delta(\Mu) ~,~\forall \Mu \in \mathcal{P} \Big\} \geq 1 - \delta,
 $$
 and therefore
 $$
  \mathbb{P}\Big\{ w^{-1} \Delta \leq \sqrt{\frac{1}{\#\mathcal{P}}\sum_{\Mu\in\mathcal{P}}\| u(\Mu)-\widetilde u(\Mu)\|_\Sigma^2} \leq w \Delta \Big\} \geq 1 - \delta.
 $$
\end{corollary}

It is important to note that Condition \eqref{coro:truth est Condition S} depends only on the cardinality of $\mathcal{P}$. This means that $K$ can be determined only knowing the number of parameters for which we need to estimate the error. Due to the logarithmic dependence of $K$ on $\#\mathcal{P}$ Corollary \ref{coro:truth est S} tolerates also (very) pessimistic estimates of $\#\mathcal{P}$.

We finish this section by noting that, using the same \emph{error residual relationship} trick \eqref{defeq:error est}, the relative errors $\|u(\Mu)- \widetilde u(\Mu)\|_\Sigma/\|u(\Mu)\|_\Sigma$ or $\sqrt{\frac{1}{\#\mathcal{P}}\sum_{\Mu\in\mathcal{P}}\| u(\Mu)-\widetilde u(\Mu)\|_\Sigma^2}/ \sqrt{\frac{1}{\#\mathcal{P}}\sum_{\Mu\in\mathcal{P}}\| u(\Mu)\|_\Sigma^2}$ can be estimated, respectively, using
\begin{equation}\label{eq:rel_err_est}
 \Delta^\text{rel}(\Mu) = \sqrt{\frac{ \frac{1}{K}\sum_{i=1}^K \big( Y_i(\Mu)^T r(\Mu) \big)^2 }{\frac{1}{K}\sum_{i=1}^K \big( Y_i(\Mu)^T f(\Mu) \big)^2 } } 
 \qquad\text{and}\qquad
 \Delta^\text{rel} = \sqrt{ \frac{\frac{1}{K\#\mathcal{P}} \sum_{\Mu \in \mathcal{P}} \sum_{k=1}^{K} \big(Y_i(\Mu)^T r(\Mu) \big)^{2}}{ \frac{1}{K\#\mathcal{P}} \sum_{\Mu \in \mathcal{P}} \sum_{k=1}^{K} \big(Y_i(\Mu)^T f(\Mu) \big)^{2} }}.
\end{equation}
We then obtain a slightly modified version of Corollary \ref{coro:truth est S}; the proof can be found in the appendix.
\begin{corollary}\label{coro:truth est S rel}
 Given $0 < \delta < 1$ and $w > e$, the condition
 \begin{equation}\label{coro:truth est Condition S rel}
   K\geq \min \left\lbrace \frac{2\log( 2 \,\#\mathcal{P}) + 2\log(\delta^{-1})}{\log(w/e)} , \enspace 3 \right \rbrace,
 \end{equation}
 is sufficient to ensure
 $$
  \mathbb{P}\Big\{ w^{-1}  \Delta^\text{rel}(\Mu) \leq \frac{\|u(\Mu)-\widetilde u(\Mu) \|_\Sigma}{\|u(\Mu)\|_{\Sigma}} \leq w  \Delta^\text{rel}(\Mu) ~,~\forall \Mu \in \mathcal{P} \Big\} \geq 1 - \delta. 
 $$
 and also 
 $$
  \mathbb{P}\Big\{ w^{-1} \Delta^\text{rel} \leq \frac{\sqrt{\frac{1}{\#\mathcal{P}}\sum_{\Mu\in\mathcal{P}}\| u(\Mu)-\widetilde u(\Mu)\|_\Sigma^2}}{\sqrt{\frac{1}{\#\mathcal{P}}\sum_{\Mu\in\mathcal{P}}\| u(\Mu)\|_\Sigma^2}} \leq w \Delta^\text{rel} \Big\} \geq 1 - \delta.
 $$
\end{corollary}

\begin{remark}[F-distribution]\label{rmk:Fdistribution}
If we have $u(\Mu)^{T}\Sigma (u(\Mu)-\widetilde u(\Mu))=0$ the random variable $(\Delta^\text{rel}(\Mu)\|u(\Mu)\|_{\Sigma})^{2}/
\| u(\Mu)- \widetilde u(\Mu)\|_\Sigma^{2}$ follows an $F$-distribution with degrees of freedom $K$ and $K$. We may then use the cumulative distribution function (cdf) of the $F$-distribution and an union bound argument to get an even more accurate lower bound for $\mathbb{P}\{ w^{-1}  \Delta^\text{rel}(\Mu) \leq \frac{\|u(\Mu)-\widetilde u(\Mu) \|_\Sigma}{\|u(\Mu)\|_{\Sigma}} \leq w  \Delta^\text{rel}(\Mu) ~,~\forall \Mu \in \mathcal{P} \}$; for further details see section \ref{sec:f-distribution}. We remark that although in general we do not have $u(\Mu)^{T}\Sigma (u(\Mu)-\widetilde u(\Mu))=0$, often $u(\Mu)$ and the error are nearly $\Sigma$-orthogonal. As a consequence we observe in the numerical experiments in Section \ref{sec5} that the empirical probability density function (pdf) of $(\Delta^\text{rel}(\Mu)\|u(\Mu)\|_{\Sigma})/\| u(\Mu)- \widetilde u(\Mu)\|_\Sigma$ follows the square root of the pdf of the $F$ distribution pretty well. 
\end{remark}

\begin{remark}[Scalar-valued QoI]
 When estimating the error in scalar-valued QoIs of the form of $s(\Mu)=l^T u(\Mu)$, the covariance matrix is $\Sigma=l\,l^T$, see subsection \ref{subsec:3.1}. In that case the random vector $Z\sim\mathcal{N}(0,\Sigma)$ follows the same distribution as $X\,l$ where $X\sim\mathcal{N}(0,1)$ is a standard normal random variable (scalar). The random dual problem \eqref{eq:randomDualProblem} then becomes
 $
  A(\Mu)^TY_i(\Mu) = X_i\, l 
 $
 and the solution is $Y_i(\Mu) = X_i \, q(\Mu)$ where $q(\Mu)$ is the solution of the deterministic dual problem $A(\Mu)^T q(\Mu) = l$. Dual problems of this form are commonly encountered for estimating linear quantities of interest, see \cite{pierce2000adjoint} for a general presentation and \cite{Haa17,RoHuPa08,zahm2017projection,AmChHu10} for the application in reduced order modeling. 
\end{remark}

\section{Approximation of the random dual solutions via the PGD}\label{sec4}

In order to obtain a fast-to-evaluate a posteriori error estimator we employ the PGD to compute approximations of the solutions of the $K$ dual problems \eqref{eq:randomDualProblem}. Given $K$ independent realizations $Z_1,\hdots,Z_K$ of $Z\sim\mathcal{N}(0,\Sigma)$, we denote by $Y(\Mu)=[Y_1(\Mu),\hdots,Y_K(\Mu)] \in\R^{N\times K}$ the horizontal concatenation of the $K$ dual solutions \eqref{eq:randomDualProblem} so that $Y(\Mu)$ satisfies the parametrized matrix equation
$$
 A(\Mu)^TY(\Mu)=[Z_1,\hdots,Z_K].
$$
We use the PGD to approximate $Y(\Mu)$ by
\begin{equation}\label{eq:dual_PGD}
 \widetilde{Y}^{L}(\Mu_{1},\hdots ,\Mu_{p}) = \sum_{m=1}^{L} Y^{m}_{\mathcal{X}} \, Y^{m}_{1} (\Mu_1)\hdots Y^{m}_{p}(\Mu_p),
\end{equation}
where the matrices $Y^{m}_{\mathcal{X}}\in\R^{N\times K}$ and the functions $Y^{m}_{i}:\mathcal{P}_i \rightarrow\R$ are built in a greedy way. The PGD algorithm we employ here is essentially the same as the one we used for the primal approximation, except with minor modifications regarding the fact that $Y(\Mu)$ is a matrix and not a vector, see subsection \ref{subsec:comput_aspects}. In subsection \ref{subsec:intertwined} we present several stopping criteria for the dual greedy enrichment, including an adaptive algorithm which builds the primal and dual PGD approximation simultaneously. 
By replacing $Y_i(\Mu)$ by $\widetilde Y_i(\Mu)$, the $i$-th column of $\widetilde Y(\Mu)$, we define pointwise fast-to-evaluate a posteriori error estimators as
\begin{equation}\label{eq: def a post est online}
 \widetilde \Delta(\Mu) := \sqrt{ \frac{1}{K} \sum_{i=1}^{K} (\widetilde Y_{i}(\Mu)^{T} r(\Mu))^{2} } \qquad \text{and} \qquad
 \widetilde{\Delta}^\text{rel}(\Mu) :=  \sqrt{\frac{ \frac{1}{K}\sum_{i=1}^K \big(\widetilde Y_i(\Mu)^T r(\Mu) \big)^2}{\frac{1}{K}\sum_{i=1}^K \big(\widetilde Y_i(\Mu)^T f(\Mu) \big)^2 } },
\end{equation}
and fast-to-evaluate error estimators for the (relative) RMS error as 
\begin{equation}\label{eq: def a post est online RMS}
 \widetilde \Delta = \sqrt{\frac{1}{\#\mathcal{P}} \sum_{\Mu\in\mathcal{P}} \widetilde \Delta(\Mu)^{2} } 
 \qquad\text{and}\qquad
\widetilde{\Delta}^\text{rel} :=  \sqrt{\frac{ \frac{1}{K\#\mathcal{P}} \sum_{\Mu \in \mathcal{P}} \sum_{i=1}^K \big(\widetilde Y_i(\Mu)^T r(\Mu) \big)^2}{\frac{1}{K\#\mathcal{P}} \sum_{\Mu\in\mathcal{P}}\sum_{i=1}^K \big(\widetilde Y_i(\Mu)^T f(\Mu) \big)^2 } }.
\end{equation}
Similarly to Proposition 3.3 in \cite{SmZaPa19}, the following propositions bound the effectivity indices for the error estimators \eqref{eq: def a post est online} and \eqref{eq: def a post est online RMS}. This will be used in subsection \ref{subsec:intertwined} to steer the dual PGD approximation. Again, these results do not depend on the condition number of $A(\Mu)$ as it is the case for the residual error estimator \eqref{eq:RelativeResNorm}.

\begin{proposition}\label{prop:dualErrorMultiplicative}
 Let $0 < \delta < 1$, $w > \sqrt{e}$ and assume
 \begin{equation}\label{eq:assumption dualErrorMultiplicative}
   K\geq \min \left\lbrace \frac{\log( \#\mathcal{P}) + \log(\delta^{-1})}{\log(w/\sqrt{e})} , \enspace 3 \right \rbrace.
 \end{equation}
 Then the fast-to-evaluate estimators $\widetilde\Delta(\mu)$ and $\widetilde \Delta$ satisfy 
 \begin{equation}\label{eq:dualErrorMultiplicative}
  \mathbb{P}\Big\{ 
  (\alpha_{\infty} w)^{-1} \widetilde\Delta(\mu) 
  \leq \|u(\mu)-\widetilde u(\mu) \|_\Sigma 
  \leq (\alpha_{\infty} w) \,\widetilde\Delta(\mu),
  \quad \mu \in \mathcal{P}
  \Big\} \geq 1-\delta 
 \end{equation}
and
 \begin{equation}\label{eq:dualErrorMultiplicative2}
  \mathbb{P}\Big\{ 
  (\alpha_{2} w)^{-1} \widetilde\Delta
  \leq \sqrt{\frac{1}{\#\mathcal{P}}\sum_{\Mu\in\mathcal{P}}\| u(\Mu)-\widetilde u(\Mu)\|_\Sigma^2} 
  \leq (\alpha_{2} w) \,\widetilde\Delta
  \Big\} \geq 1-\delta, 
 \end{equation}
where
 \begin{equation}\label{eq:alpha}
  \alpha_{\infty} := \, \max_{\mu \in \mathcal{P}} \left( \max \left\{\frac{\Delta(\mu)}{\widetilde \Delta(\mu)} \,,\, \frac{\widetilde \Delta(\mu)}{\Delta(\mu)} \right\} \right) \geq 1 \qquad \text{ and } \qquad \alpha_{2}:= \,  \max \left\{\frac{\Delta}{\widetilde \Delta} \,,\, \frac{\widetilde \Delta}{\Delta} \right\} \geq 1.
 \end{equation}
\end{proposition}
\begin{proof}
The proof follows exactly the proof of Proposition 3.3 in \cite{SmZaPa19} and we include it in the appendix for the sake of completeness.
\end{proof}

\begin{proposition}\label{prop:dualErrorMultiplicativerel}
 Let $0 < \delta < 1$, $w > e$ and assume
 \begin{equation}\label{eq:assumption dualErrorMultiplicativerel}
   K\geq \min \left\lbrace \frac{2\log(  2\,\#\mathcal{P}) + 2\log(\delta^{-1})}{\log(w/e)} , \enspace 3 \right \rbrace,
 \end{equation}
 Then the fast-to-evaluate estimators $\widetilde\Delta^\text{rel}(\Mu)$ and $\widetilde \Delta^\text{rel}$ satisfy
 \begin{equation}\label{eq:dualErrorMultiplicativerel}
  \mathbb{P}\Big\{ 
  (\alpha_{\infty}^\text{rel} w)^{-1} \widetilde\Delta^\text{rel}(\Mu) 
  \leq \frac{\|u(\Mu)-\widetilde u(\Mu) \|_\Sigma }{\|u(\Mu)\|_\Sigma }
  \leq (\alpha_{\infty}^\text{rel} w) \,\widetilde\Delta^\text{rel}(\Mu),
  \quad \Mu \in \mathcal{P}
  \Big\} \geq 1-\delta 
 \end{equation}
and 
 \begin{equation}\label{eq:dualErrorMultiplicativerel2}
  \mathbb{P}\Big\{ 
  (\alpha_{2}^\text{rel} w)^{-1} \widetilde\Delta^\text{rel}
  \leq \frac{\sqrt{\frac{1}{\#\mathcal{P}}\sum_{\Mu\in\mathcal{P}}\| u(\Mu)-\widetilde u(\Mu)\|_\Sigma^2}}{\sqrt{\frac{1}{\#\mathcal{P}}\sum_{\Mu\in\mathcal{P}}\| u(\Mu)\|_\Sigma^2}}
  \leq (\alpha_{2}^\text{rel} w) \,\widetilde\Delta^\text{rel}
  \Big\} \geq 1-\delta, 
 \end{equation}
where
 \begin{equation}\label{eq:alpharel}
  \alpha_{\infty}^\text{rel} := \, \max_{\Mu \in \mathcal{P}} \left( \max \left\{\frac{\Delta^\text{rel}(\Mu)}{\widetilde \Delta^\text{rel}(\Mu)} \,,\, \frac{\widetilde \Delta^\text{rel}(\Mu)}{\Delta^\text{rel}(\Mu)} \right\} \right) \geq 1 \qquad \text{ and } \qquad \alpha_{2}^\text{rel}:= \,  \max \left\{\frac{\Delta^\text{rel}}{\widetilde \Delta^\text{rel}} \,,\, \frac{\widetilde \Delta^\text{rel}}{\Delta^\text{rel}} \right\} \geq 1.
 \end{equation}
\end{proposition}
\begin{proof}
Can be proved completely analogously to Propostion \ref{prop:dualErrorMultiplicative}.
\end{proof}

\subsection{PGD algorithm for the random dual solution}\label{subsec:comput_aspects}

To construct the PGD approximation $\widetilde{Y}^{L}(\Mu)$ \eqref{eq:dual_PGD}, we employ a pure greedy algorithm. All terms in \eqref{eq:dual_PGD} are constructed one after the other: after $L-1$ iterations, the rank-one correction $Y_{\mathcal{X}}^L Y_{1}^L(\Mu_1)  \cdots  Y_{p}^L(\Mu_p)$ is computed as the solution to 
\begin{equation}\label{eq:MinResPGD_4dual}
\min_{Y_{\mathcal{X}}\in\R^{N\times K}} \,
 \min_{Y_{1}:\mathcal{P}_1\rightarrow\R} 
 \cdots
 \min_{Y_{p}:\mathcal{P}_p\rightarrow\R} 
 ~\sum_{i=1}^K\sum_{\Mu\in\mathcal{P}} \|A(\Mu)^T\Big(\widetilde{Y}^{L-1}_i(\Mu) + Y_{\mathcal{X},i} Y_{1}(\Mu_1)  \cdots  Y_{p}(\Mu_p)\Big) - Z_i\|_{2}^2,
\end{equation}
where $\widetilde{Y}^{L-1}_i(\Mu)$ and $Y_{\mathcal{X},i}$ denote the $i$-th columns of the matrices $\widetilde{Y}^{L-1}_i(\Mu)$ and $Y_{\mathcal{X}}$ respectively. 
Formulation \eqref{eq:MinResPGD_4dual} corresponds to the \emph{Minimal residual PGD} as it consists in minimizing the sum of the $K$ residuals associated with the random dual problems \eqref{eq:randomDualProblem}. If the matrix $A(\Mu)$ is SPD, one can also use the \emph{Galerkin PGD}  formulation
\begin{equation}\label{eq:MinEnergyPGD_4dual}
\min_{Y_{\mathcal{X}}\in\R^{N\times K}} \,
 \min_{Y_{1}:\mathcal{P}_1\rightarrow\R} 
 \cdots
 \min_{Y_{p}:\mathcal{P}_p\rightarrow\R} 
 ~\sum_{i=1}^K\sum_{\Mu\in\mathcal{P}} \|A(\Mu)^T\Big(\widetilde{Y}^{L-1}_i(\Mu) + Y_{\mathcal{X},i} Y_{1}(\Mu_1)  \cdots  Y_{p}(\Mu_p)\Big) - Z_i\|_{A(\Mu)^{-1}}^2 ,
\end{equation}
where the canonical norm $\|\cdot\|_2$ has been replaced by the energy norm $\|\cdot\|_{A(\Mu)^{-1}}$. As explained before, both of these formulations can be solved efficiently using the \emph{alternating least squares}.

Comparing \eqref{eq:MinResPGD_4dual} or \eqref{eq:MinEnergyPGD_4dual} with the formulations used for the primal PGD solution \eqref{eq:MinResPGD} or \eqref{eq:MinEnergyPGD}, the major difference is that the solution is matrix-valued for $K>1$, so that the first minimization problem over $Y_{\mathcal{X}}\in\R^{N\times K}$ might be, at a first glance, more expensive to compute compared to the minimum over $u_{\mathcal{X}}\in\R^N$ in \eqref{eq:MinResPGD} or \eqref{eq:MinEnergyPGD}. However, the minimization problem over $Y_{\mathcal{X}}\in\R^{N\times K}$ possesses a particular structure which we may exploit for computational efficiency. This structure is that each column of $Y_{\mathcal{X}}$ is the solution to a least squares problem with the same operator but with different right-hand side. Indeed, for fixed $Y_1(\Mu_1),\hdots,Y_p(\Mu_p)$, the $i$-th column of $Y_{\mathcal{X}}$ minimizes
$$
 Y_{\mathcal{X},i} \mapsto \sum_{\Mu\in\mathcal{P}} \|A(\Mu)^T Y_{\mathcal{X},i} Y_{1}(\Mu_1)  \cdots  Y_{p}(\Mu_p) - \big(Z_i-A(\Mu)^T\widetilde{Y}^{L-1}_i(\Mu) \big)\|_{*}^2,
$$
with either $\|\cdot\|_*=\|\cdot\|_2 $ or $\|\cdot\|_*=\|\cdot\|_{A(\Mu)^{-1}}$. 
This means that $Y_{\mathcal{X},i}$ can be computed by solving a linear system of the form $\widetilde A_{\mathcal{X}} Y_{\mathcal{X},i} = \widetilde Z_i$ for some operator $\widetilde A_{\mathcal{X}} \in\R^{N\times N}$ and right-hand side $\widetilde Z_i\in\R^N$ which are assembled\footnote{For example with the Galerkin PGD formulation \eqref{eq:MinEnergyPGD_4dual}, we have $\widetilde A_{\mathcal{X}} = \sum_{\Mu\in\Pcal} A(\Mu)^T Y_{1}(\Mu_1)^2  \cdots  Y_{p}(\Mu_p)^2$ and $\widetilde Z_i =\sum_{\Mu\in\Pcal} ( Z_i - A(\Mu)^T\widetilde{Y}^{L-1}_i(\Mu)) Y_{1}(\Mu_1)  \cdots  Y_{p}(\Mu_p) $.} using $A(\Mu)^T$, $Z_i$ and also $Y_{1}(\Mu_1),\hdots,Y_{p}(\Mu_p)$ and $\widetilde{Y}^{L-1}_i(\Mu)$. 
Then, by concatenating the $K$ right-hand sides, the matrix $Y_{\mathcal{X}}$ solution to \eqref{eq:MinResPGD_4dual} or \eqref{eq:MinEnergyPGD_4dual} can be obtained by solving a matrix equation of the form
$$
 \widetilde A_{\mathcal{X}} Y_{\mathcal{X}} = [\widetilde Z_1,\hdots,\widetilde Z_K].
$$
For certain solvers $Y_{\mathcal{X}}$ can be computed at a complexity which hardly depends on $K$. For instance, one can precompute say a sparse LU or Cholesky factorization of $\widetilde A_{\mathcal{X}}$ and use this factorization to solve for the multiple right-hand sides. Thus, when $K\ll N$, the complexity is dominated by the factorization of $\widetilde A_{\mathcal{X}}$ so that the costs for computing the dual PGD approximation is at least comparable to the one for the primal PGD approximation.

\begin{remark}[Complexity comparison with the hierarchical error estimator]\label{rmk:RankDualVSRankStag}
 Thanks to the above remark, it is worth mentioning that the cost for computing the dual approximation $\widetilde Y^L(\Mu)$ is essentially the same as for computing a hierarchical estimate $\widetilde u^{M+k}(\Mu)$ with $k=L$. Indeed, both $\widetilde Y^L(\Mu)$ and $\widetilde u^{M+k}(\Mu)$ require the computation of $k=L$ PGD updates (primal or dual) which can be computed at comparable cost. Therefore, in the numerical experiments, we pay particular attention in comparing our randomized estimators \eqref{eq: def a post est online} or \eqref{eq: def a post est online RMS} with the stagnation-based error estimator \eqref{eq:StagnationErrorEstimator} with $k=L$ to ensure a fair comparison.
 
\end{remark}

\begin{remark}
 Instead of looking for $\widetilde{Y}^{L}(\Mu)$ as in \eqref{eq:dual_PGD}, one could have build on the tensor product structure of $\R^{N\times K}=\R^N \otimes \R^K$ and look for approximation on the form of $Y_i(\Mu)\approx \sum_{m=1}^{L} Y^{m}_{\mathcal{X}} \,Y^{m}_K(i)\, Y^{m}_{1} (\Mu_1)\hdots Y^{m}_{p}(\Mu_p)$ where $Y^{m}_{\mathcal{X}}\in\R^N$ is a vector (not a matrix) and where $Y^{m}_K:\{1,\hdots,K\}\rightarrow\R$. This amounts to consider the column index $i$ as an additional parameter, and thus to increase the order of the problem by one. This option is however not pursued in this paper. 
 
\end{remark}

\subsection{Primal-dual intertwined PGD algorithm}\label{subsec:intertwined}

In this subsection we discuss how to choose the rank $L$ of the dual approximation $\widetilde Y^L(\Mu)$. 

First, one could choose the rank \emph{based on a given (computational) budget}. Assuming for instance that the budget is tight one could only consider dual PGD approximations of say rank $1$ to $5$. The numerical experiments in section \ref{sec5} show that this choice works remarkably well even for high dimensional examples. 
This choice, however, does not allow controlling the quality of the dual approximation so that neither $\alpha_{2/\infty}$ nor $\alpha_{2/\infty}^\text{rel}$ in Propositions \ref{prop:dualErrorMultiplicative} and \ref{prop:dualErrorMultiplicativerel} can be bounded.

In order to control the error in the dual approximation, we suggest building the primal and dual approximations simultaneously in an intertwined manner as described in Algorithm \ref{algo:pgf_greedy}. We focus here on the relative error estimator $\Delta^\text{ref}$ to simplify the notation, but other errors such as $\Delta$ or $\max_{\Mu \in \Pcal} \Delta(\Mu)$ can be considered as well. 
First, we calculate $K$ in line \ref{algo:calc_K} for given effectivity $\alpha w$ and failure probability $\delta$ via \eqref{eq:assumption dualErrorMultiplicativerel}.
As $\widetilde \Delta^\text{ref}$ invokes Proposition \ref{prop:Chi2Tail} $2\#\Pcal$-times in each greedy iteration and the maximal number of greedy iterations is $M_{max}$, we need to apply Proposition \ref{prop:Chi2Tail} at most $M_{max}2\#\Pcal$-times during Algorithm \ref{algo:pgf_greedy} and thus have to replace $2\#\Pcal$ by $M_{max}2\#\Pcal$ in \eqref{eq:assumption dualErrorMultiplicativerel} in the calculation of $K$. We then draw $K$ independent Gaussian random vectors with mean zero and covariance matrix $\Sigma$. 

Next, we describe how to build $\widetilde u^M(\Mu)$ and $\widetilde Y^L(\Mu)$ in an intertwined manner. Within each iteration of Algorithm \ref{algo:pgf_greedy} we first compute  the \emph{primal} rank-one correction using \eqref{eq:MinResPGD} or \eqref{eq:MinEnergyPGD}. Algorithm \ref{algo:pgf_greedy} terminates either if $\widetilde \Delta^\text{rel} \leq tol$ or $M=M_{max}$.

However, due to the dual PGD approximation it is a priori not clear whether $\widetilde \Delta^\text{rel}$ has an effectivity close to unity and thus estimates the relative error well. An effectivity that is well below one might result in an premature termination of the algorithm, which is clearly undesirable. This is why every time we need to estimate the error, we first assess the quality of $\widetilde \Delta^\text{rel}$ before using it as a stopping criterion and then, if necessary, enrich the dual PGD approximation until the quality of $\widetilde \Delta^\text{rel}$ is satisfactory.

\begin{algorithm}[t]
 \caption{Intertwined greedy construction of primal and dual PGD approximation}\label{algo:pgf_greedy}
 \begin{algorithmic}[1]
\State{\textbf{INPUT: tolerance $tol$, failure probability $\delta$, maximal rank $M_{max}$, $w$, $\alpha$, $\Pcal$, number of increments $k$}}
\State{Calculate $K$ based on \eqref{eq:assumption dualErrorMultiplicativerel} replacing $2\#\Pcal$ by $2M_{max}\#\Pcal$}\label{algo:calc_K}
\State{Draw $K$ independent random Gaussian vectors with mean zero and covariance matrix $\Sigma$}
\State{Initialize $\ur^0(\Mu) = 0$, $M=0$, $\widetilde \Delta^\text{rel} = 2 tol$, $L=1$}\\
\While{$\widetilde \Delta^\text{rel} > tol$ and $M < M_{max}$}
\State{\emph{Enrich primal PGD approximation:}}
\State{$M\leftarrow M+1$}
\State{Compute the primal rank-one correction using \eqref{eq:MinResPGD} or \eqref{eq:MinEnergyPGD}}
\State{Update $\widetilde{u}^{M}(\Mu) = \widetilde{u}^{M-1}(\Mu) + u^{M}_{\mathcal{X}}  u^{M}_{1}(\Mu_1)  \hdots  u^{M}_{p}(\Mu_p)$}\\

\State{\emph{Enrich dual PGD approximation until tolerance reached:}}
\State{Compute $\alpha_{2,k}^\text{rel}$ as in \eqref{eq:alpha2k}}
 \While{$\alpha_{2,k}^\text{rel}> \alpha$}\label{algo:while}
 \State{$L\leftarrow L+1$} 
 \State{Compute the primal rank-one correction using \eqref{eq:MinResPGD_4dual} or \eqref{eq:MinEnergyPGD_4dual}}
 \State{Update $\widetilde{Y}^{L}(\Mu) = \widetilde{Y}^{L-1}(\Mu) + Y^{L}_{\mathcal{X}} Y^{L}_{1}(\Mu_1)  \hdots Y^{L}_{p}(\Mu_p)$}

 \State{Compute $\alpha_{2,k}^\text{rel}$ as in \eqref{eq:alpha2k}}
 \EndWhile \\
 
 \State{\emph{Compute the error estimator:}}
 \State{Update $\widetilde \Delta^\text{rel}$ as in \eqref{eq: def a post est online RMS}}
 \EndWhile \\
 
\State{\textbf{OUTPUT:} $\ur^{M}(\Mu)$ and $\widetilde \Delta^\text{rel}$}
 \end{algorithmic}
\end{algorithm}

Motivated by Proposition \ref{prop:dualErrorMultiplicativerel}, the natural way to measure the quality of $\widetilde \Delta^\text{rel}$ is $\alpha_2^\text{rel} = \max\{\Delta^\text{rel}/\widetilde \Delta^\text{rel} ; \widetilde \Delta^\text{rel}/\Delta^\text{rel}\}$. Since the unbiased error estimator $\Delta^\text{rel}$ is not computationally feasible, as the solution $u(\Mu)$ is not known in advance, one possibility is to replace $\Delta^\text{rel}$ with the computationally feasible error estimator
\begin{equation}
 \Delta^\text{rel}_{+k} = \sqrt{ \frac{ \sum_{\Mu \in \mathcal{P}} \sum_{k=1}^{K} \big(Z_i^T (\widetilde u^{M+k}(\Mu)-\widetilde u^{M}(\Mu) )\big)^{2}}{  \sum_{\Mu \in \mathcal{P}} \sum_{k=1}^{K} \big(Z_i^T \widetilde u^{M+k}(\Mu) \big)^{2} }}.
\end{equation}
Here, the number of additional increments $k$ are chosen depending on the available computational budget. 

We also introduce
\begin{equation}
\widetilde\Delta^\text{rel}_{+k} = \frac{\sqrt{\sum_{\Mu\in\mathcal{P}} \sum_{i=1}^{K}(\widetilde Y_{i}^L(\Mu)^{T} A(\Mu)(\widetilde u^{M+k}(\Mu)-\widetilde u^{M}(\Mu)))^{2}}}{\sqrt{\sum_{\Mu\in\mathcal{P}}\sum_{i=1}^{K}(\widetilde Y_{i}^L(\Mu)^{T} A(\Mu)\widetilde u^{M+k}(\Mu))^{2}}},
\end{equation}
which is the error estimator of the \emph{known} increment $\widetilde u^{M+k}(\Mu)-\widetilde u^{M}(\Mu)$.
For larger $k$ we expect $\widetilde u^{M+k}(\Mu)-\widetilde u^{M}(\Mu)$ to be already a quite good error model justifying the choice $\widetilde\Delta^\text{rel}_{+k} = \widetilde \Delta^\text{rel}$. We recall and emphasize that the stagnation-based or hierarchical error estimator does neither ensure that the respective estimator bounds the error or allows to bound the effectivity unless we have for instance a uniform bound for the convergence rate of the PGD; see subsection \ref{subsec:PGD}. Therefore, we use the increments solely as an error model to train the dual PGD approximation such that we can then rely on Proposition \ref{prop:dualErrorMultiplicativerel}. 
We then utilize in line \ref{algo:while} 
\begin{equation}\label{eq:alpha2k}
 \alpha_{2,k}^\text{rel} := \max \left\{\frac{\Delta^\text{rel}_{\pm k}}{\widetilde \Delta^\text{rel}_{\pm k}} \,;\, \frac{\widetilde \Delta^\text{rel}_{\pm k}}{\Delta^\text{rel}_{\pm k}} \right\} ,
\end{equation}
where 
\begin{equation}\label{eq:less_exp_err_ind}
 \Delta^\text{rel}_{-k} = \sqrt{ \frac{ \sum_{\Mu \in \mathcal{P}} \sum_{k=1}^{K} \big(Z_i^T (\widetilde u^{M}(\Mu)-\widetilde u^{M-k}(\Mu) )\big)^{2}}{  \sum_{\Mu \in \mathcal{P}} \sum_{k=1}^{K} \big(Z_i^T \widetilde u^{M}(\Mu) \big)^{2} }}
 \quad\text{ and }\quad
 \widetilde\Delta^\text{rel}_{-k} = \sqrt{ \frac{ \sum_{\Mu \in \mathcal{P}} \sum_{k=1}^{K} \big(\widetilde Y_i^L(\Mu)^T A(\Mu)(\widetilde u^{M}(\Mu)-\widetilde u^{M-k}(\Mu)) \big)^{2}}{  \sum_{\Mu \in \mathcal{P}} \sum_{k=1}^{K} \big(\widetilde Y_i^L(\Mu)^T A(\Mu)\widetilde u^{M}(\Mu) \big)^{2} }}.
\end{equation}
With this choice, one can compute $\alpha_{2,k}^\text{rel}$ with \emph{no additional} computational effort. Apart from that, the motivation behind the definition of $\Delta^\text{rel}_{-k}$ and $\widetilde\Delta^\text{rel}_{-k}$ is that we can estimate the norms of the increments $\widetilde u^{M}(\Mu)-\widetilde u^{M-k}(\Mu)$ well. Assuming a relatively uniform convergence behavior of the primal PGD approximation, it can thus be conjectured that $\widetilde \Delta^\text{rel}$ then also estimates the norm of $u(\Mu) - \widetilde u^{M}(\Mu)$ well; this is demonstrated in section \ref{sec5}. Algorithm \ref{algo:pgf_greedy} terminates the enrichment of the dual PGD approximation within each iteration if $\alpha_{2,k}^\text{rel}$ falls below some prescribed value $\alpha>1$. 

\begin{remark}[Independence]
We note that if the error in the primal PGD approximation is already close to the target tolerance $tol$ the primal PGD approximation can become dependent on the respective realizations of $Z_{1},\hdots,Z_{K}$ in the following sense: For some realizations the estimator $\widetilde \Delta^\text{rel}$ lies below $tol$ and thus terminates Algorithm \ref{algo:pgf_greedy} while for others we have $\widetilde \Delta^\text{rel} >tol$ and thus obtain a primal PGD approximation of higher rank. A careful analysis of this effect on the failure probabilities in Proposition \ref{prop:dualErrorMultiplicativerel} is however not necessary as by replacing $2\#\mathcal{P}$ by $2M_{max}\#\Pcal$ in the calculation of $K$ in line \ref{algo:calc_K} we ensure that the number of samples $K$ is chosen large enough such that we can estimate all possible outcomes. 
\end{remark}

\section{Numerical experiments}\label{sec5}

We demonstrate the randomized error estimator for two test cases which are derived from the benchmark problem introduced in \cite{lam2018multifidelity}. The source code is freely available in Matlab$^{\circledR}$ at the address\footnote{\url{https://gitlab.inria.fr/ozahm/randomdual_pgd.git}} so that all numerical results presented in this section are entirely reproducible.
We consider the system of linear PDEs whose solution $\mathcal{u}:D\rightarrow\R^2$ is a vector field satisfying
\begin{equation}\label{eq:WrenchmarkModel}
 - \text{div}(C:\varepsilon(\mathcal{u})) - k^2 \mathcal{u} = \mathcal{f},
\end{equation}
where $D\subset\R^2$ is a domain which has the shape of a wrench; see Fig.~\ref{fig:Wrench}. Here, $\varepsilon(\mathcal{u})=\frac{1}{2}(\nabla\mathcal{u}+\nabla\mathcal{u}^T)$ is the strain tensor, $k^2$ the wave number, and $C$ the fourth-order stiffness tensor derived from the plane stress assumption with Young's modulus $E$ and fixed Poisson coefficient $\nu=0.3$ such that
\begin{equation}\label{eq:Hooke}
 C:\varepsilon = \frac{E}{1+\nu}\varepsilon + \frac{\nu E}{1-\nu^2}\text{trace}(\varepsilon)I_2.
\end{equation}
The right hand-side $\mathcal{f}$ and the boundary conditions are represented in Fig.~\ref{fig:Wrench}. The finite element discretization using piecewise affine basis functions $\psi_1,\hdots,\psi_N$ yields $N=1467$ degrees of freedom. In the remainder of the paper, we only consider estimating errors in the natural Hilbert norm $\|\cdot\|_{\mathcal{X}}$ defined by 
$$
 \|\mathcal{u}\|_{\mathcal{X}}^2 = \int_{D} \text{trace}(\nabla \mathcal{u}(x)^T\nabla \mathcal{u}(x)) + \|\mathcal{u}(x)\|_2^2 \,\text{d}x,
$$
so that $\Sigma = R_{\mathcal{X}_N}$, as explained section \ref{sec2} (see equation \eqref{eq:NormsDefinition}).

\begin{figure}[t]
  \centering 
  \begin{subfigure}[t]{0.43\textwidth}
  \centering 
    \includegraphics[width = \textwidth]{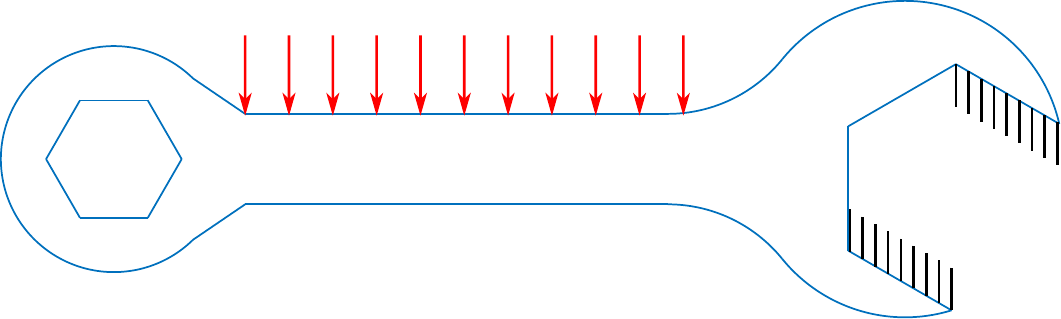} 
  \end{subfigure}~~~~~~~~~
  \begin{subfigure}[t]{0.43\textwidth}
  \centering
    \includegraphics[width = \textwidth]{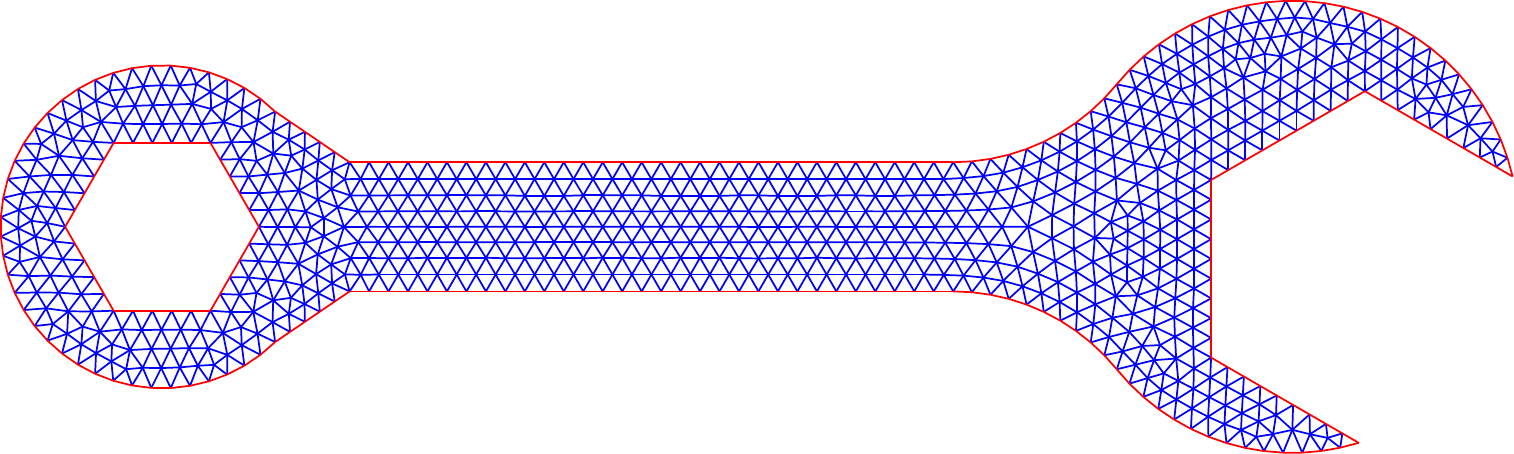}
  \end{subfigure}
  \caption{Left: Geometry of the wrench. Dirichlet condition $\mathcal{u}=0$ (black chopped lines), vertical unitary linear forcing $\mathcal{f}$ (red arrows). Right: Mesh for the finite element approximation.
  }
  \label{fig:Wrench}
\end{figure}

\begin{figure}[t]
  \centering 
  \begin{subfigure}[t]{0.43\textwidth}
  \centering 
    \includegraphics[width = 0.7\textwidth]{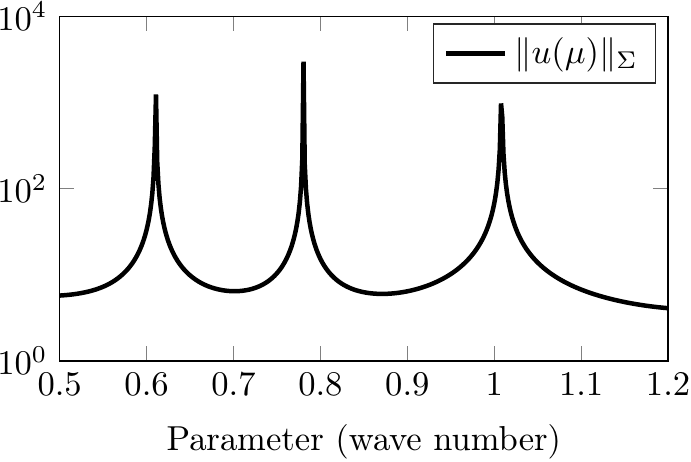} 
  \end{subfigure}~~~~~~~~~
  \begin{subfigure}[t]{0.43\textwidth}
  \centering
    \includegraphics[width = \textwidth]{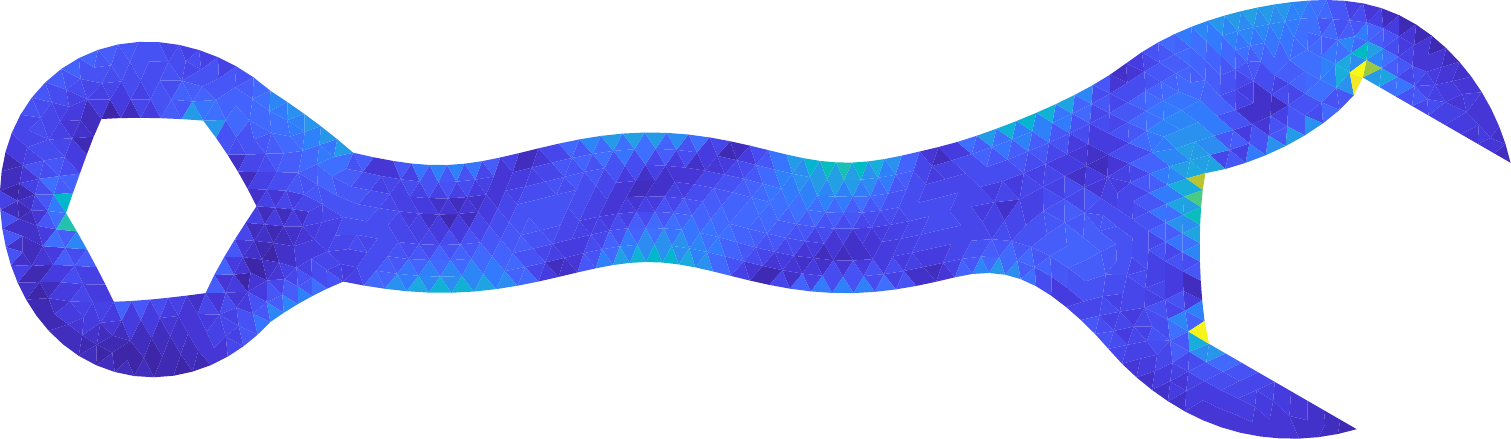}
  \end{subfigure}
  \caption{Time-harmonic elastodynamics. Left: Norm of the solution on $\mathcal{P}$. Right: finite element solution $\mathcal{u}(\Mu)$ for $\Mu=1$. The color represents the VonMises stress.
  }
  \label{fig:Helmholtz}
\end{figure}

\subsection{Non-homogeneous time-harmonic elastodynamics problem}

In this benchmark problem the Young's modulus is fixed to $E=1$ uniformly over the domain $D$ and the wave number $k^2$ varies between $0.5$ and $1.2$. The scalar parameter $\Mu=k^2$ varies on the parameter set $\mathcal{P}$ defined as the uniform grid of $[0.5 , 1.2]$ containing $\#\mathcal{P}=500$ points. The parametrized operator derived from \eqref{eq:WrenchmarkModel} is 
$$
 \mathcal{A}(\Mu) = - \text{div}(C:\varepsilon(\cdot)) - \Mu \, (\cdot).
$$
The parametrized stiffness matrix $A(\Mu)\in\R^{N\times N}$, obtained via a FE discretization, admits an affine decomposition $A(\Mu) = A_1 - \Mu A_2$, where $(A_1)_{ij} = \int_D \varepsilon(\psi_i):C:\varepsilon(\psi_j) \text{d}x$ and $(A_2)_{ij} = \int_D \psi_i^T \psi_j \text{d}x$. As this benchmark problem contains only one scalar parameter, it is convenient to inspect how the method behaves within the parameter set. We observe in Fig.~\ref{fig:Helmholtz} that the norm of the solution blows up around three parameter values, which indicates resonances of the operator. This shows that for this parameter range, the matrix $A(\Mu)$ is not SPD and therefore we use the Minimal Residual PGD formulations \eqref{eq:MinResPGD} and \eqref{eq:MinResPGD_4dual} for the primal and dual PGD approximation.

\subsubsection{Estimating the relative error}

\begin{figure}[t]
  \centering 
  \begin{subfigure}[t]{0.48\textwidth}
  \centering 
    \includegraphics[width = \textwidth]{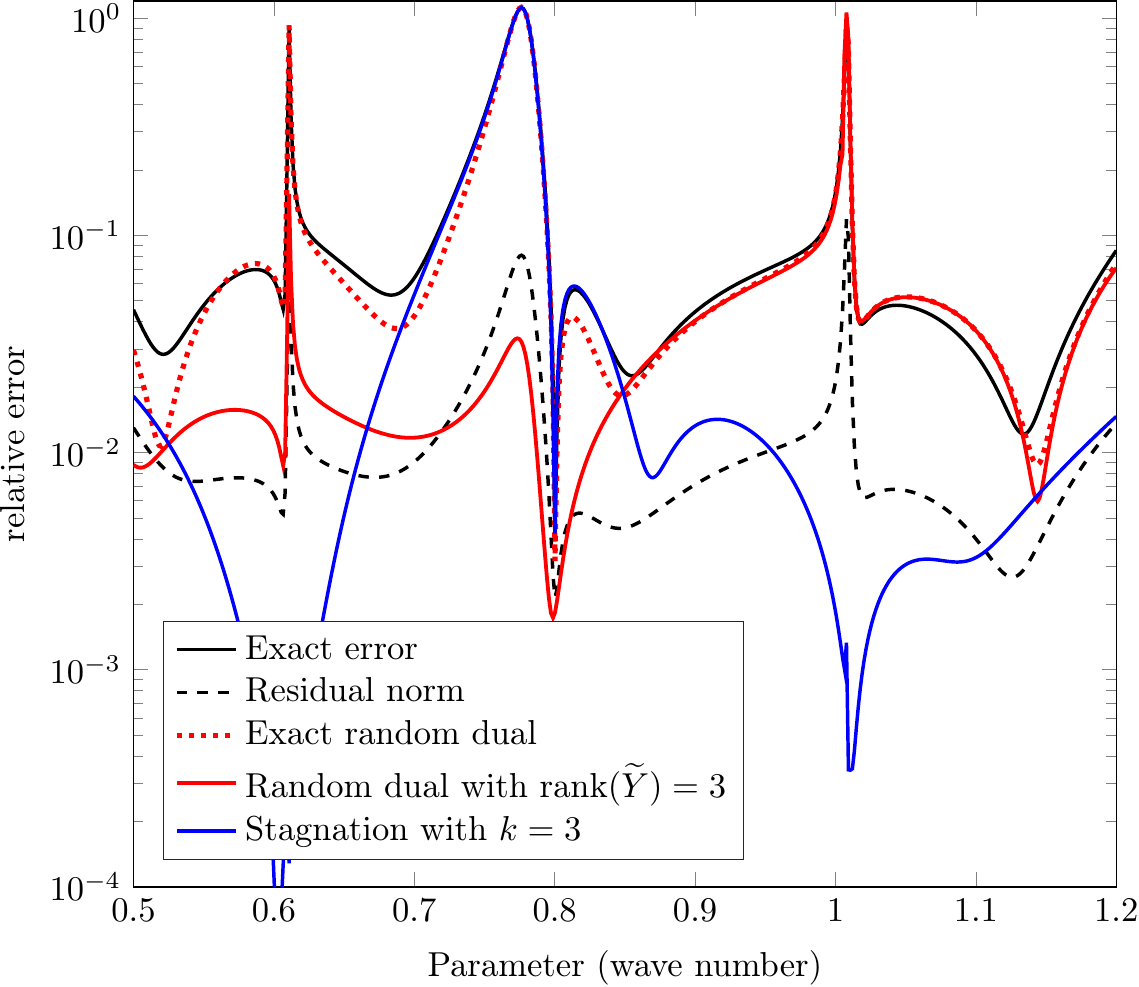} 
  \end{subfigure}~~~~~~~~~
  \begin{subfigure}[t]{0.48\textwidth}
  \centering
    \includegraphics[width = \textwidth]{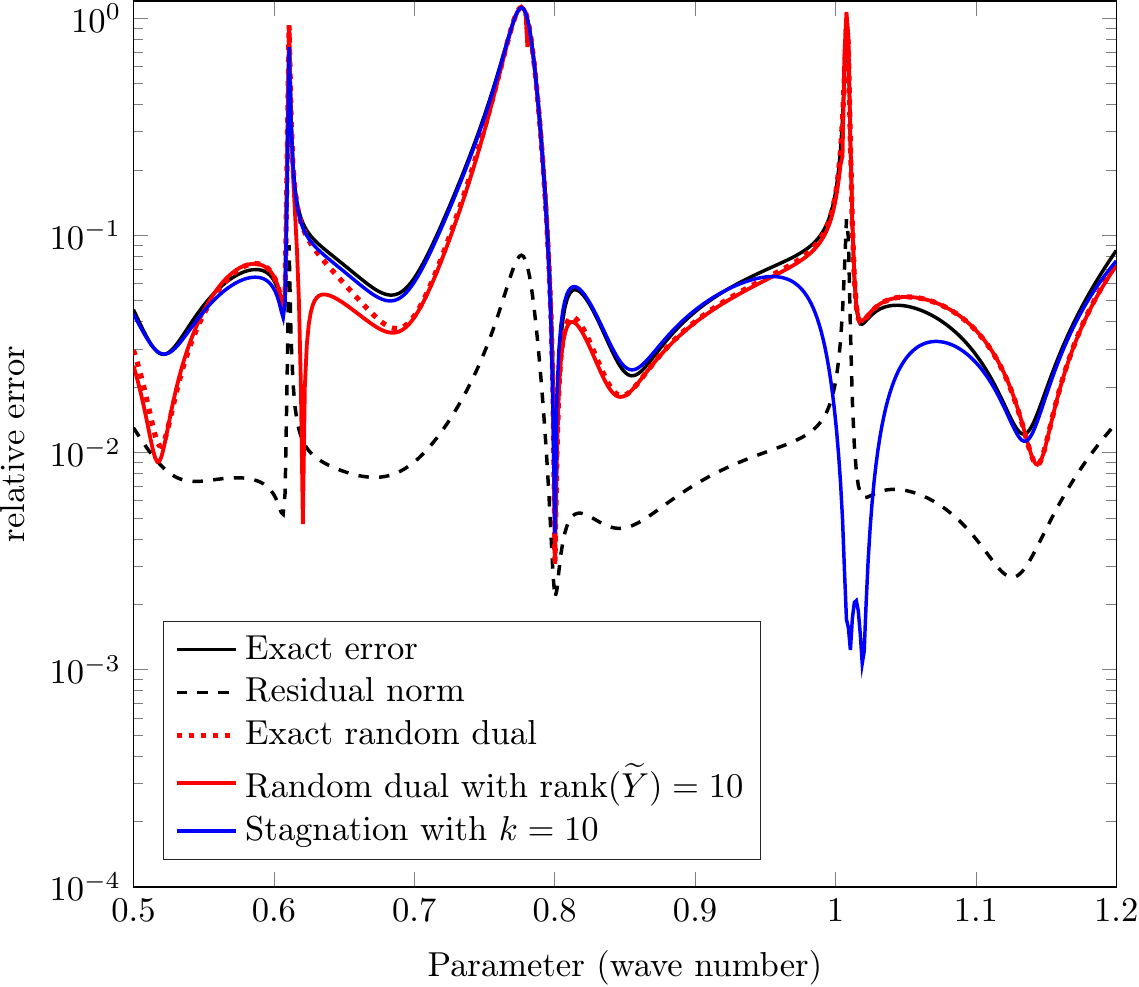}
  \end{subfigure}
  \caption{Time-harmonic elastodynamics. Comparison of several error estimators for the relative error of an PGD approximation with rank ten. Solid black line: Exact error \eqref{eq:HelmholtzExactRelativeError}. Dashed black line: $\Delta^\text{res}_{10}(\Mu)$. Dotted red line: one realization of the exact relative random estimate $\Delta^\text{rel}(\Mu)$ with $K=6$, see Equation \eqref{eq:rel_err_est}. Solid red line: approximation $\widetilde{\Delta}^\text{rel}(\Mu)$ of $\Delta^\text{rel}(\Mu)$ defined in \eqref{eq: def a post est online} $\text{rank}(\widetilde Y)=3$ (left) and $\text{rank}(\widetilde Y)=10$ (right). Solid blue line: $\Delta^\text{stag}_{10,k}(\Mu)$ with either $k=3$ (left) or $k=10$ (right).}
  \label{fig:Helmholtz_ErrorsOnParameterSpace}
\end{figure}

We first compute a PGD approximation $\widetilde{u}^M(\Mu)$ with rank $M=10$ and we aim at estimating the relative error
\begin{equation}\label{eq:HelmholtzExactRelativeError}
 \frac{\|u(\Mu) - \widetilde{u}^{10}(\Mu) \|_{\mathcal{X}_N}}{\|u(\Mu) \|_{\mathcal{X}_N}}
\end{equation}
for every $\Mu\in\mathcal{P}$. In Fig.~\ref{fig:Helmholtz_ErrorsOnParameterSpace} we show a comparison of the performances of the randomized residual-based error estimator $\widetilde{\Delta}^\text{rel}(\Mu)$ defined in \eqref{eq: def a post est online}, the error estimator based on the dual norm of the residual, and the stagnation-based error estimator defined, respectively, as
\begin{equation}\label{eq:HelmholtzResidualAndStagnationEstimators}
\Delta^\text{res}_{M}(\Mu):=\frac{\|A(\Mu)\widetilde{u}^{M}(\Mu)-f(\Mu)\|_{\mathcal{X}_N'}}{\|f(\Mu)\|_{\mathcal{X}_N'} }
 \qquad\text{and}\qquad
\Delta^\text{stag}_{M,k}(\Mu):= \frac{\| \widetilde{u}^{M+k}(\Mu) - \widetilde{u}^{M}(\Mu) \|_{\mathcal{X}_N}}{\|\widetilde{u}^{M+k}(\Mu) \|_{\mathcal{X}_N}}.
\end{equation}
First, we observe in Fig.~\ref{fig:Helmholtz_ErrorsOnParameterSpace} that although the residual-based error estimator $\Delta^\text{res}_{10}(\Mu)$ follows the overall variations in the error, it underestimates the error significantly over the whole considered parameter range. Following Remark \ref{rmk:RankDualVSRankStag}, in order to have a fair comparison between the stagnation-based and the randomized error estimator, we set $k=L=\text{rank}(\widetilde Y^L)$ so that the two estimators require the same computational effort in terms of number of PGD corrections that have to be computed. We observe that with either $k=L=3$ (left) or $k=L=10$ (right), the randomized residual-based error estimator is closer to the true relative error compared to the stagnation-based estimator $\Delta^\text{stag}_{10,k}(\Mu)$. Notice however that the randomized error estimator $\widetilde{\Delta}^\text{rel}(\Mu)$ is converging towards $\Delta^\text{rel}(\Mu)$ with $\text{rank}(\widetilde Y)$ which is not the true relative error (see the dashed red lines on Fig.~\ref{fig:Helmholtz_ErrorsOnParameterSpace}), whereas $\Delta^\text{stag}_{10,k}(\Mu)$ will converge to the true relative error with $k$. Nevertheless, we do observe that for small $k=L$, our random dual is always competitive. In addition, increasing the number of samples $K$ will reduce the variance of both $\widetilde{\Delta}^\text{rel}(\Mu)$ and $\Delta^\text{rel}(\Mu)$ such that the randomized error estimator will follow the error even more closely (see also  Fig.~\ref{fig:Helmholtz_histOfEffectivities} to that end). We finally note that Fig.~\ref{fig:Helmholtz_ErrorsOnParameterSpace} shows only one (representative) realization of $\widetilde{\Delta}^\text{rel}(\Mu)$ and $\Delta^\text{rel}(\Mu)$, and the behavior of the effectivity of $\widetilde{\Delta}^\text{rel}(\Mu)$ over various realizations will be discussed next.

\begin{figure}[t]
  \centering 
  \includegraphics[width = \textwidth]{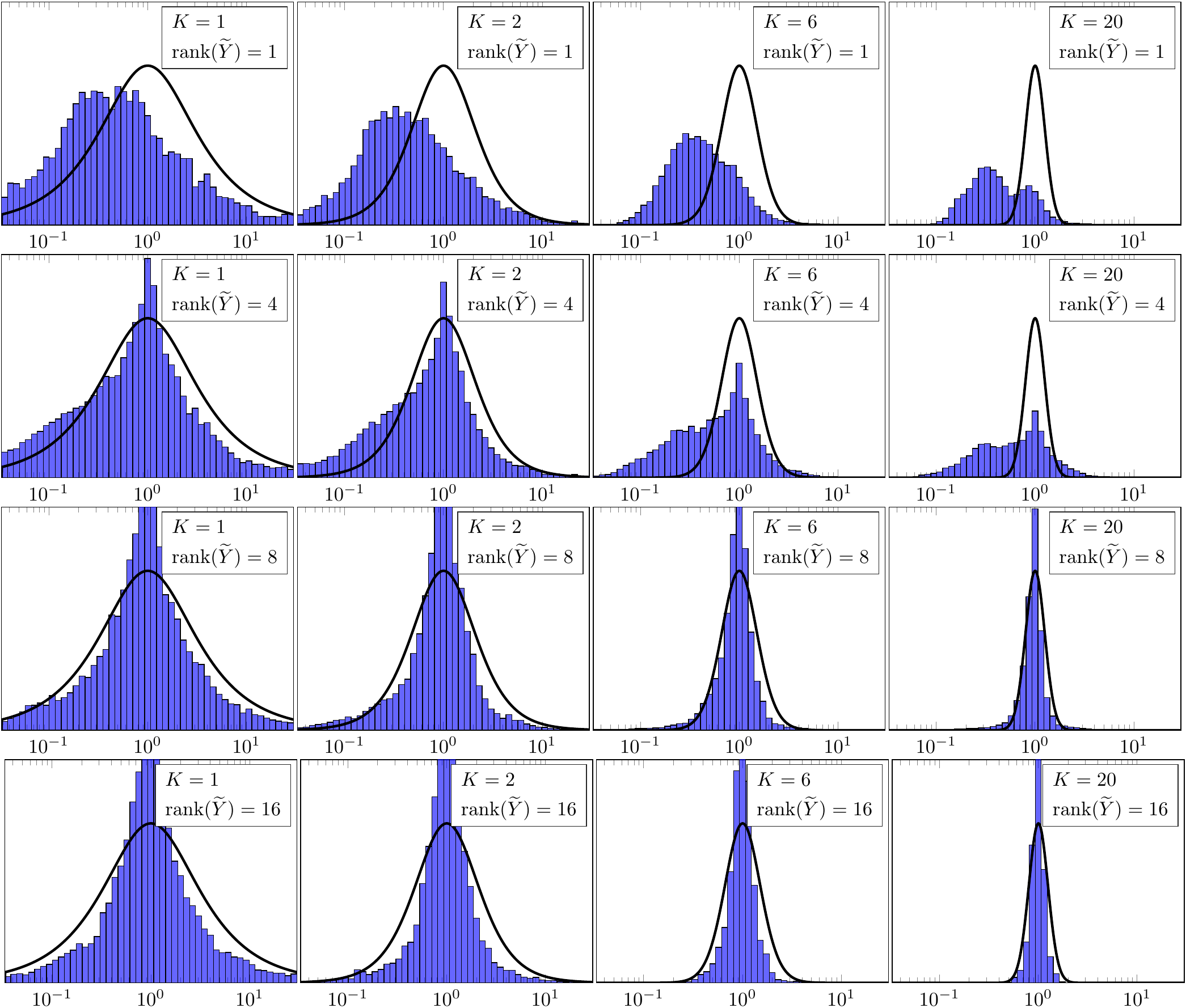}
  \caption{Time-harmonic elastodynamics. Histograms of $100$ realizations of the effectivity index $\{\eta(\Mu), \Mu\in\mathcal{P}\}$ defined by \eqref{eq:Helmholtz_ETA} for $K=1,2,6,20$ and $\text{rank}(\widetilde Y) = 1,4,8,16$. Solid line: pdf of the square root of the F-distribution with degrees of freedom $K$ and $K$.
  }
  \label{fig:Helmholtz_histOfEffectivities}
\end{figure}

In detail, we first draw $100$ different realizations of $Z=[Z_1,\hdots,Z_K]$ , $Z_i\sim\mathcal{N}(0,\Sigma)$, for $K=1,2,6,20$ and compute the corresponding dual solutions $\widetilde Y^L$ with either $L=\text{rank}(\widetilde Y^L) = 1,4,8,16$. Next, we compute the randomized error estimator $\widetilde{\Delta}^\text{rel}(\Mu)$ and calculate the effectivity index
\begin{equation}\label{eq:Helmholtz_ETA}
 \eta(\Mu) = \frac{ \widetilde{\Delta}^\text{rel}(\Mu) }{\|u(\Mu) - \widetilde{u}^{10}(\Mu) \|_{\mathcal{X}_N}/\|u(\Mu)\|_{\mathcal{X}_N}},
\end{equation} 
for all $\Mu\in\mathcal{P}$. The resulting histograms are shown in Fig.~\ref{fig:Helmholtz_histOfEffectivities}. First, we emphasize that especially for $\widetilde{\Delta}^\text{rel}(\Mu)$ based on a PGD approximation of the dual problems of rank $8$, $16$, and to some extent even $4$, we clearly see a strong concentration of $\eta(\Mu)$ around one as desired. If the number of samples $K$ increases we observe, as expected, that the effectivity of $\widetilde{\Delta}^\text{rel}(\Mu)$ is very close to one (see for instance picture with $K=20$ and $\rank(\widetilde Y)=8$ in Fig.~\ref{fig:Helmholtz_histOfEffectivities}). We recall that for the primal PGD approximation used for these numerical experiments we chose the rank $M=10$. We thus highlight and infer that by using a dual PGD approximation of lower rank than those of the primal PGD approximation, we can obtain an error estimator with an effectivity index very close to one. 

Using a dual PGD approximation of rank one and to some extend rank four, we observe in Fig.~\ref{fig:Helmholtz_histOfEffectivities} that the mean over all samples of $\eta(\Mu)$ is significantly smaller than one (see for instance picture with $K=6$ and $\rank(\widetilde Y)=1$ in Fig.~\ref{fig:Helmholtz_histOfEffectivities}). This behavior can be explained by the fact that in these cases the dual PGD approximation is not rich enough to characterize all the features of the error which results in an underestimation of the error. 

Finally, as explained in Remark \ref{rmk:Fdistribution} and more detailed in section \ref{sec:f-distribution}, $\eta(\Mu)$ is distributed as the square root of an $F$-distribution with degress of freedom $(K,K)$, provided some orthogonality condition holds. The solid black line in Fig.~\ref{fig:Helmholtz_histOfEffectivities} is the pdf of such a square root of an $F$-distribution with degress of freedom $(K,K)$. In particular, we observe a nice accordance with the histograms showing that, for sufficiently large $L=\text{rank}(\widetilde Y)$, the effectivity indices concentrate around unity with $K$, as predicted by Proposition \ref{prop:dualErrorMultiplicativerel}.

\subsubsection{Testing the intertwined algorithm}

 \begin{figure}[t]
  \centering 
  \begin{subfigure}[t]{0.48\textwidth}
  \centering 
    \includegraphics[width = \textwidth]{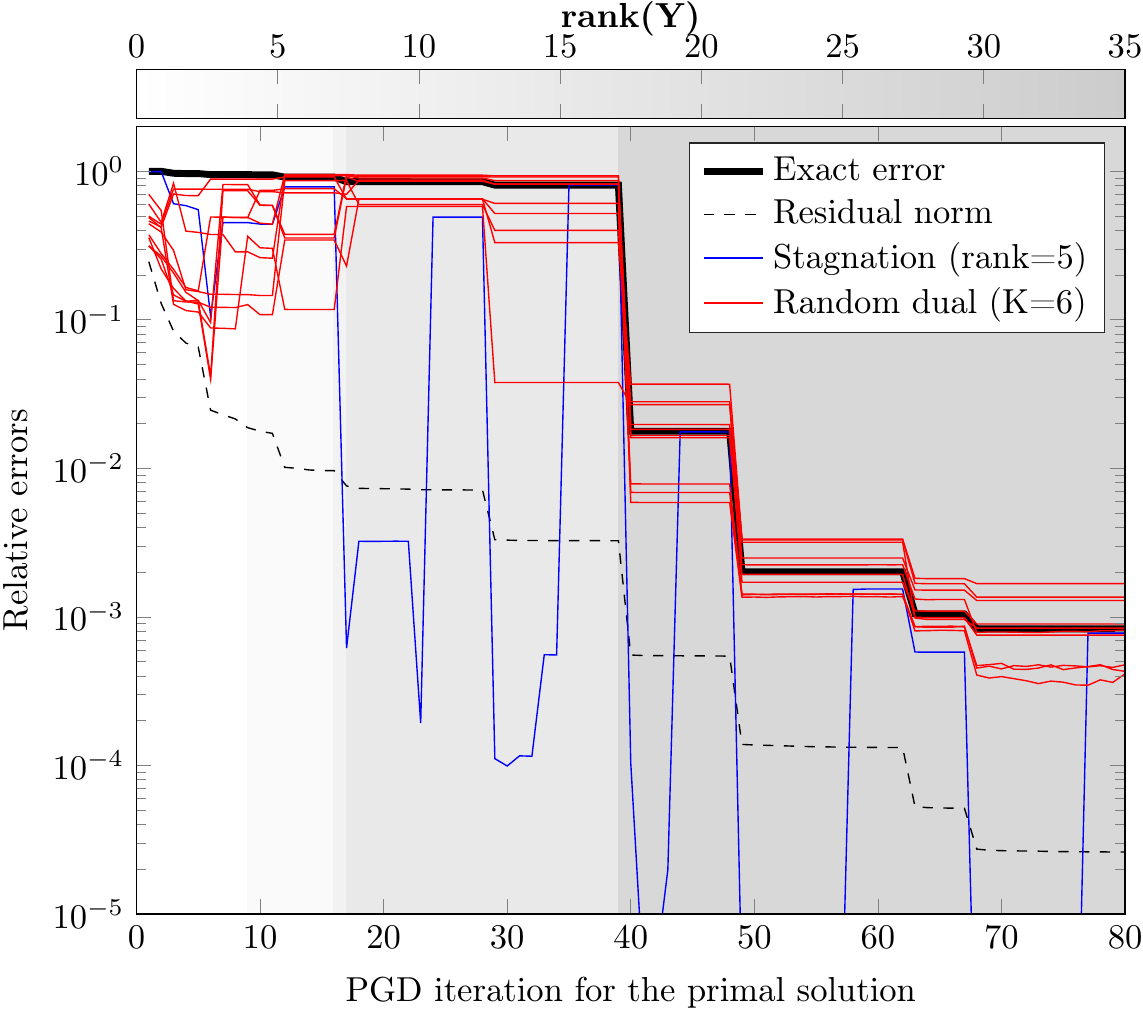} 
  \end{subfigure}~~~~~~
  \begin{subfigure}[t]{0.48\textwidth}
  \centering
    \includegraphics[width = \textwidth]{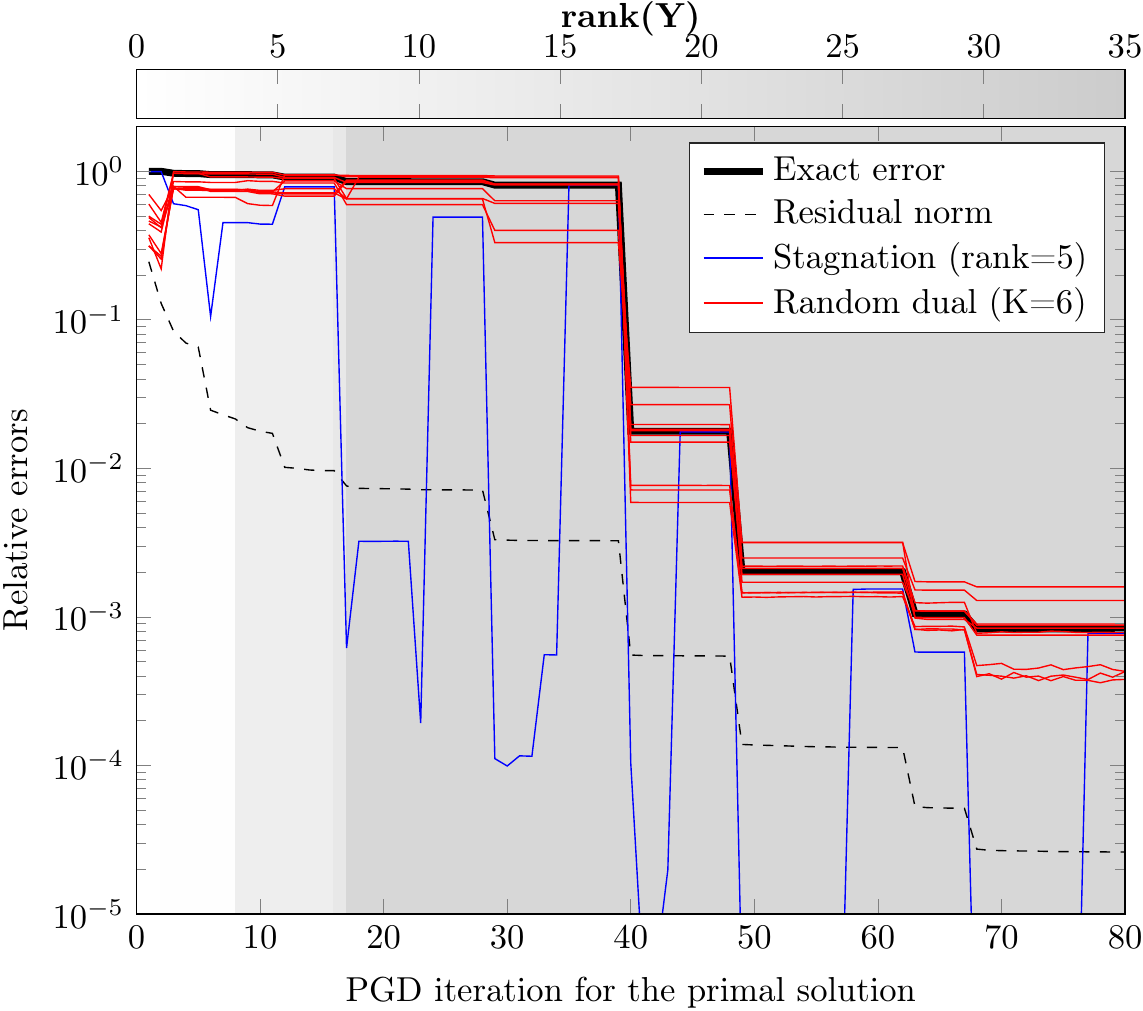}
  \end{subfigure}
  \caption{Time-harmonic elastodynamics. Illustration of the intertwined algorithm \ref{algo:pgf_greedy}. 
  Red lines: error estimator $\widetilde\Delta^\text{rel}$ produced by Algorithm \ref{algo:pgf_greedy} with $k=6$ and with either $\alpha=5$ (left) or $\alpha=2$ (right). Each of the ten red lines corresponds to a different realization of the $Z_1,\hdots,Z_K$ with $K=10$.
  Solid black line: exact relative RMS error $\|u-\widetilde u^M(\Mu)\|/\|u(\Mu)\|$.
  Blue line: stagnation-based error estimator $\Delta_{M,k}^\text{stag}$ defined by \eqref{eq:HelmholtzResidualAndStagnationEstimators_Intertwined} with $k=5$.
  Dashed black lines: relative RMS residual norm $\Delta_{M}^\text{res}$, see \eqref{eq:HelmholtzResidualAndStagnationEstimators_Intertwined}.
  Background color: value of $L=\text{rank}(\widetilde Y^L)$ averaged over the ten realizations.}
  \label{fig:Helmholtz_intertwined}
\end{figure}

We now illustrate the intertwined algorithm introduced in Subsection \ref{subsec:intertwined}. Here, the goal is to monitor the convergence of the PGD algorithm during the iteration process $M=1,2,\hdots$. We thus estimate the relative RMS error
$$
 \frac{\|u-\widetilde{u}^M \|}{\|u\|} \overset{\eqref{eq:IdealPGDError}}{=} 
 \sqrt{\frac{\frac{1}{\#\mathcal{P}}\sum_{\Mu\in\mathcal{P}} \|u(\Mu)-\widetilde{u}^M(\Mu)\|_{\mathcal{X}_N}^2 }{ \frac{1}{\#\mathcal{P}}\sum_{\Mu\in\mathcal{P}} \|u(\Mu)\|_{\mathcal{X}_N}^2 }},
$$
with our relative randomized error estimator $\widetilde{\Delta}^\text{rel}(\Mu)$ introduced in \eqref{eq: def a post est online} with the rank adaptation for the random dual approximation $\widetilde Y^L$. We employ the economical error indicators within the intertwined algorithm \ref{algo:pgf_greedy} with $k=6$.
Fig.~\ref{fig:Helmholtz_intertwined} shows the performance of our estimator (10 realizations) as well as the increase in $L=\text{rank}(\widetilde Y^L)$ (averaged over the 10 realizations) during the PGD iteration process. We also compare to the relative RMS residual norm $\Delta_{M}^\text{res}$ \eqref{eq:RelativeResNorm} and to the relative RMS stagnation-based error estimator $\Delta_{M,k}^\text{stag}$ \eqref{eq:StagnationErrorEstimator} with $k=5$ for which we recall the expression
\begin{equation}\label{eq:HelmholtzResidualAndStagnationEstimators_Intertwined}
 \Delta_{M}^\text{res} := \sqrt{\frac{\frac{1}{\#\mathcal{P}}\sum_{\Mu\in\mathcal{P}}\|A(\Mu)\widetilde{u}^{M}(\Mu)-f(\Mu)\|^2_{\mathcal{X}_N'} }{\frac{1}{\#\mathcal{P}}\sum_{\Mu\in\mathcal{P}}\|f(\Mu)\|^2_{\mathcal{X}_N'}}}
 \qquad\text{and}\qquad
 \Delta_{M,5}^\text{stag} :=\sqrt{\frac{\frac{1}{\#\mathcal{P}}\sum_{\Mu\in\mathcal{P}}\| \widetilde{u}^{M+5}(\Mu) - \widetilde{u}^{M}(\Mu) \|^2_{\mathcal{X}_N} }{\frac{1}{\#\mathcal{P}}\sum_{\Mu\in\mathcal{P}}\|\widetilde{u}^{M+5}(\Mu) \|^2_{\mathcal{X}_N} }}
\end{equation}

We notice that our estimator is much more stable compared to the stagnation-based error estimator which tends to dramatically underestimate the error during the phases where the PGD is plateauing. Our estimator is also much more accurate compared to $\Delta_{M}^\text{res}$. When comparing the results with $\alpha=5$ (left) and $\alpha=2$ (right), we observe that a large $\alpha$ results in smaller dual rank $L=\text{rank}(\widetilde Y^L)$ but the quality of the error indicator is deteriorated. Reciprocally, a smaller $\alpha$ yields more demanding dual approximations so that $L=\text{rank}(\widetilde Y^L)$ is, on average, larger. Let us notice that a rank explosion of the dual variable $\widetilde Y^L$ is not recommended because the complexity for evaluating $\widetilde \Delta^\text{rel}$ will also explode. This will be further discussed in an upcoming paper.

\subsection{Parametrized linear elasticity with high-dimensional parameter space} 

\begin{figure}[t]
  \centering 
  \begin{subfigure}[t]{0.43\textwidth}
  \centering 
    \includegraphics[width = \textwidth]{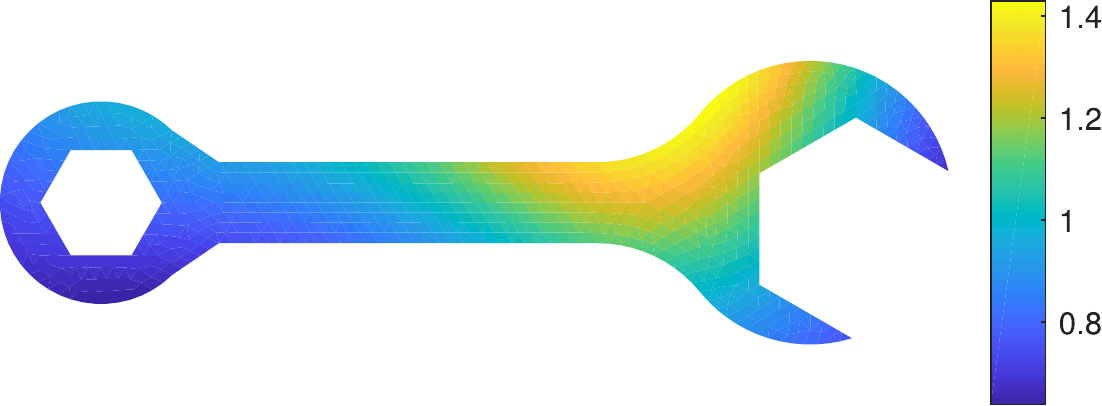} 
  \end{subfigure}~~~~~~~~~
  \begin{subfigure}[t]{0.43\textwidth}
  \centering
    \includegraphics[width = \textwidth]{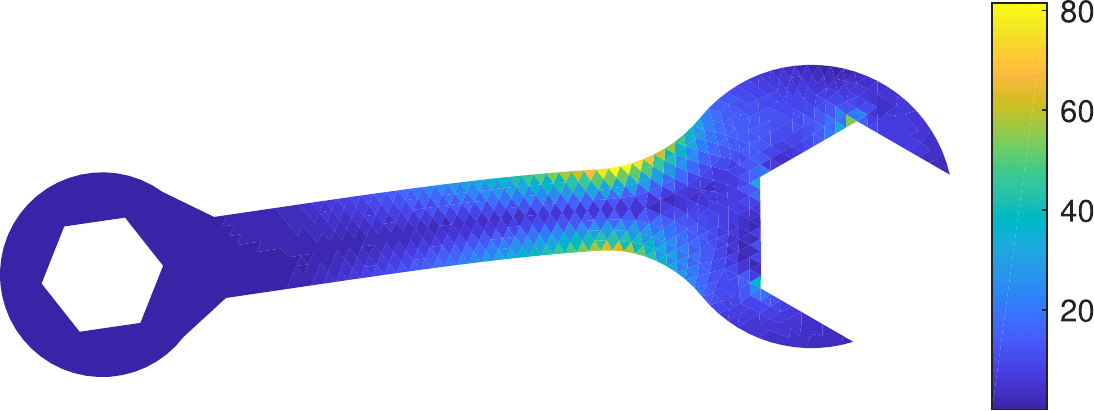}
  \end{subfigure}
  \caption{Linear elasticity. Left: One realization of the field $\log(E)$. Right: Corresponding solution $\mathcal{u}(\Mu)$. The color represents the VonMises stress.
  }
  \label{fig:Elliptic}
\end{figure}

In this benchmark problem the wave number is set to $k^2=0$ so that \eqref{eq:WrenchmarkModel} becomes an elliptic PDE. Young's modulus $E$ is a random field on $D$ that is log-normally distributed so that $\log(E)$ is a Gaussian random field on $D$ with zero-mean and covariance function $cov(x,y) =  \sigma_0 \exp( -(\|x-y\|_2/l_0)^2  )$ with $\sigma=0.4$ and $l_0=4$. A realization of the Young's modulus field is shown in Fig.~\ref{fig:Elliptic}. Since $E$ is a random field, such a parametrized problem contains infinitely many parameters. We use here a truncated Karhunen Lo\`eve decomposition with $p=20$ terms to approximate the random Young's modulus field by a finite number of parameter $\Mu_1,\hdots,\Mu_{20}$, that is
\begin{equation}\label{eq:KLdecomp}
 E \approx E(\Mu_1,\hdots,\Mu_{20}) = \exp( \sum_{i=1}^{20} \Mu_i \sqrt{\sigma_i}\phi_i ),
\end{equation}
where $(\sigma_i,\phi_i)$ is the $i$-th eigenpair of the covariance function $cov$ and $\Mu_i$ are independent standard Gaussian random variables. This means that the parameter set is $\R^{20}$ and that $\rho$ is the density of a standard normal distribution in dimension $p=20$. In actual practice, we replace the parameter set by a $20$-dimensional grid $\mathcal{P}:=(x_1,\hdots,x_{50})^{20} \subset \R^{20}$, where $x_1,\hdots,x_{50}$ is an optimal quantization of the univariate standard normal distribution using $50$ points which are obtained by Lloyd's algorithm \cite{du1999centroidal}. Next, we use the Empirical Interpolation Method (EIM) \cite{BMNP04} using $28$ ``magic points'' in order to obtain an affine decomposition of the form 
\begin{equation}\label{eq:EIM}
 E(\Mu_1,\hdots,\Mu_{20}) \approx \sum_{i=1}^{28} E_j \,\zeta^j(\Mu_1,\hdots,\Mu_{20}).
\end{equation}
for some spatial functions $E_j:D\rightarrow\R$ and parametric functions $\zeta^j:\R^{20}\rightarrow\R$. With the EIM method, the functions $\zeta^j$ are defined by means of magic point $x^j\in D$ by $\zeta^j(\Mu) = \exp( \sum_{i=1}^{20} \Mu_i \sqrt{\sigma_i}\phi_i(x^j) ) = \prod_{i=1}^{20}\zeta_i^j(\Mu_i)$, where $\zeta_i^j(\Mu_i)=\exp( \Mu_i\sqrt{\sigma_i}\phi_i(x^j) )$.
Together with the truncated Karhunen Lo\`eve decomposition \eqref{eq:KLdecomp} and with the EIM \eqref{eq:EIM} the approximation of the Young's modulus fields does not exceed $1\%$ of error, and the differential operator is given by
$$
 \mathcal{A}(\Mu) = - \text{div}(C(\Mu):\varepsilon(\cdot)) 
 \approx - \sum_{j=1}^{28} \text{div}(K_j:\varepsilon(\cdot)) ~ \zeta_1^j(\Mu_1)\hdots\zeta_{20}^j(\Mu_{20}) ,
$$
where the forth-order stiffness tensors $C_j$ are given as in \eqref{eq:Hooke} with $E$ replaced by $E_j$.  The parametrized stiffness matrix $A(\Mu)$, obtained via a FE discretization, admits an affine decomposition $A(\Mu) = \sum_{i=1}^{28} A_i \zeta_1^j(\Mu_1)\hdots\zeta_{20}^j(\Mu_{20})$, where $A_i\in\R^{N\times N}$ is the matrix resulting from the discretization of the continuous PDE operator $\text{div}(K_j:\varepsilon(\cdot))$. As $A(\Mu)$ is a parametrized SPD matrix, we use the \emph{Galerkin PGD} formulation \eqref{eq:MinEnergyPGD} and \eqref{eq:MinEnergyPGD_4dual} to compute the primal and dual PGD solutions $\widetilde{u}^{M}(\Mu)$ and $\widetilde{Y}^{L}(\Mu)$, respectively.

For this high-dimensional benchmark problem our goal is to monitor the convergence of the PGD. Again, we estimate the relative RMS error $\|u-\widetilde{u}^M \|/\|u\|$ using $\widetilde{\Delta}^\text{rel}_M$ as in \eqref{eq: def a post est online} where we fix $K=3$ and the rank of the approximate dual $\widetilde Y^L$ as $L=\text{rank}(\widetilde Y^L)=1,3,5$. The results are reported in Fig.~\ref{fig:EllipticMonitorPGD}. We highlight that, remarkably, a very small dual rank $L$ is enough to obtain a very accurate estimation of the error. In particular, our approach outperforms the stagnation based error estimator, which is very unstable, and also the relative RMS residual norm, which is almost one order of magnitude off. As in the previous subsection, we mention that the evaluation of the estimator $\widetilde\Delta^\text{rel}$ can be numerically expensive due to high primal and dual ranks ($M$ and $L$).

\begin{figure}[t]
  \centering 
  \includegraphics[width = \textwidth]{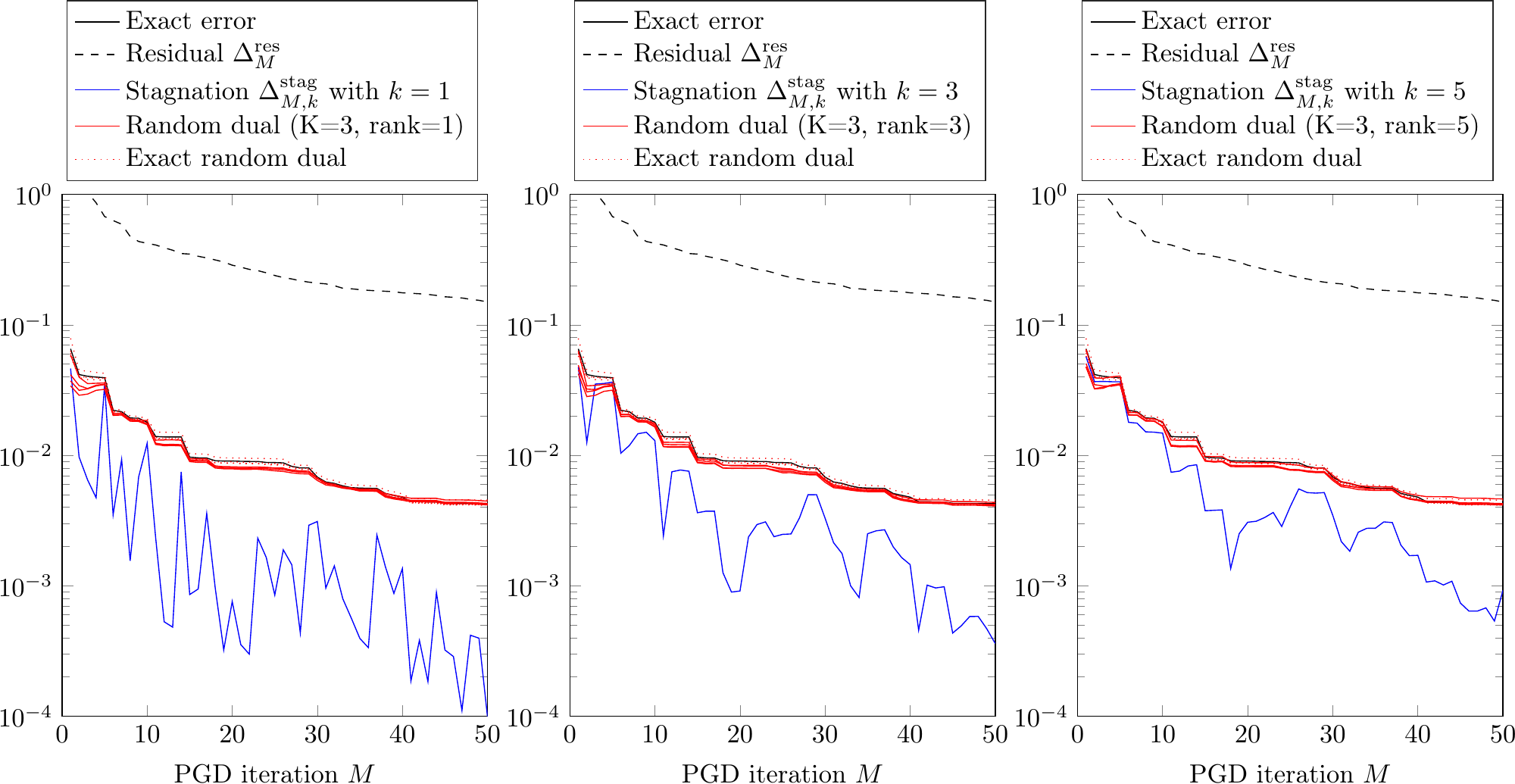} 
  \caption{Linear elasticity. Monitoring the PGD convergence: Comparison of the residual-based error $\Delta_{M}^\text{res}$ defined in \eqref{eq:RelativeResNorm} (dashed black curves), the stagnation-based estimator $\Delta_{M,k}^\text{stag}$ defined in \eqref{eq:StagnationErrorEstimator} (blue lines), and the estimator $\widetilde\Delta^\text{rel}$ defined by \eqref{eq: def a post est online RMS} with $K=3$ and with $L=\text{rank}(\widetilde Y^L)$ either equal to 1 (left), 3 (middle), or 5 (right). The dashed red lines corresponds to the exact random dual estimator $\Delta^\text{rel}$.
  }
  \label{fig:EllipticMonitorPGD}
\end{figure}

\section{Conclusions}\label{sec6}

In this paper we proposed a novel a posteriori error estimator for the PGD, extending the concept introduced in \cite{SmZaPa19} to this setting. The proposed randomized error estimator does not require the estimation of stability constants, its effectivity is close to unity and lies within an interval chosen by the user at specified high probability. Moreover, the suggested framework allows estimating the error in any norm chosen by the user and the error in some (vector-valued) QoI.
To derive the error estimator we make use of the concentration phenomenon of Gaussian maps. Exploiting the error residual relationship and approximating the associated random dual problems via the PGD yields a fast-to-evaluate a posteriori error estimator. To ensure that also the effectivity of this fast-to-evaluate error estimator lies within a prescribed interval at high probability we suggested building the primal and dual PGD approximations in a intertwined manner, where we enrich the dual PGD approximation in each iteration until the error estimator reaches the desired quality. 

The numerical experiments for a parametrized time-harmonic elastodynamics problem and a parametrized linear elasticity problem with $20$-dimensional parameter space show that very often and even for a very tight effectivity interval it is sufficient to consider a dual PGD approximation that has a significantly lower rank than the primal PGD approximation. Thanks to its favorable computational complexity, the costs for the proposed a posteriori error estimator are about the same or much smaller than the costs for the primal PGD approximation. A numerical comparison with an error estimator defined as the dual norm of the residual and a stagnation-based error estimator (also called hierarchical error estimator) shows that the randomized error estimator provides a much more precise description of the behavior of the error at a comparable computational cost.

\appendix

\section{Proofs}

\begin{proof}[Proof of Corollary \ref{coro:truth est S rel}]
Assuming $w > \sqrt{e}$ and $K\geq3$ we can use Proposition \ref{prop:Chi2Tail} for any vector $v(\Mu) \in \R^{N}$ and a fixed parameter $\Mu \in \mathcal{P}$ and obtain
\begin{align}\label{eq:aux_prop}
\mathbb{P}\Big\{ w^{-1}\|v(\Mu)\|_\Sigma \leq \| \Phi v(\Mu) \|_2  \leq w  \|v(\Mu)\|_\Sigma \Big\}  \geq 1-\Big( \frac{\sqrt{e}}{w} \Big)^{K}. 
\end{align}
Using the definition of $\Delta^\text{rel}(\Mu)$, invoking \eqref{eq:aux_prop} twice and using a union bound argument the following inequality
\begin{align}\label{eq:aux_coro}
\frac{1}{w^{2}} \Delta^\text{rel}(\Mu) = \frac{1}{w^{2}} \frac{\|\Phi (u(\Mu) - \widetilde u(\Mu)\|_{2}}{\|\Phi u(\Mu) \|_{2}} \leq \frac{1}{w} \frac{\|u(\Mu) - \widetilde u(\Mu)\|_{\Sigma}}{\|\Phi u(\Mu) \|_{2}} \leq \frac{\|u(\Mu) - \widetilde u(\Mu)\|_{\Sigma}}{\| u(\Mu) \|_{\Sigma}} \leq w \frac{\|u(\Mu) - \widetilde u(\Mu)\|_{\Sigma}}{\|\Phi u(\Mu) \|_{2}} \leq w^{2} \Delta^\text{rel}(\Mu)
\end{align}
holds at least with probability  $1-2\Big( \frac{\sqrt{e}}{w} \Big)^{K}$. Using again a union bound argument yields
\begin{align*}
 &\mathbb{P}\Big\{ \frac{1}{w^{2}} \Delta^\text{rel}(\Mu) \leq \frac{\|u(\Mu) - \widetilde u(\Mu)\|_{\Sigma}}{\| u(\Mu) \|_{\Sigma}} \leq \frac{1}{w^{2}} \Delta^\text{rel}(\Mu) ~,~\forall \Mu \in\mathcal{P} \Big\} \geq 1-2(\# \mathcal{P})  \Big( \frac{\sqrt{e}}{w} \Big)^{K}. 
\end{align*}
For a given $0<\delta<1$ condition 
$
  K\geq (\log( 2 \# \mathcal{P}) + \log(\delta^{-1}))/\log(w/\sqrt{e}),
$
is equivalent to $1-(2 \# \mathcal{P})( \frac{\sqrt{e}}{w} )^{K} \geq 1-\delta$ and ensures that \eqref{eq:aux_coro} holds for all $\Mu \in\mathcal{P}$ with probability larger than $1-\delta$. Setting $\hat{w}=w^{2}$ and renaming $\hat{w}$ as $w$ concludes the proof for $\Delta^\text{rel}(\Mu)$; the proof for $\Delta^\text{rel}$ follows similar lines.
\end{proof}

\begin{proof}[Proof of Proposition \ref{prop:dualErrorMultiplicative}]
We provide the proof for $\Delta(\Mu)$ noting that the proof for $\Delta$ can be done all the same. By Corollary \ref{coro:truth est S}, it holds with probability larger than $1-\delta$ that $w^{-1} \Delta(\Mu) \leq \|u(\Mu)-\widetilde u(\Mu) \|_\Sigma \leq w \Delta(\Mu)$ for all $\Mu\in\mathcal{P}$. Then with the same probability we have
$$
 \|u(\Mu)-\widetilde u(\Mu) \| 
 \leq  w \Delta(\Mu) 
 \leq  w \left( \sup_{\Mu'\in\mathcal{P}} \frac{\Delta(\Mu')}{\widetilde \Delta(\Mu')} \right) \widetilde \Delta(\Mu) 
 \overset{\eqref{eq:alpha}}{\leq} (\alpha_{\infty} w) \widetilde \Delta(\Mu),
$$
and
$$
 \|u(\Mu)-\widetilde u(\Mu) \| 
 \geq w^{-1} \Delta(\Mu) 
 \geq w^{-1} \left( \inf_{\Mu'\in\mathcal{P}} \frac{\Delta(\Mu')}{\widetilde \Delta(\Mu')} \right) \widetilde \Delta(\Mu) 
 \overset{\eqref{eq:alpha}}{\geq} (\alpha_{\infty} w)^{-1} \widetilde \Delta(\Mu),
$$
for any $\Mu\in\mathcal{P}$, which yields \eqref{eq:dualErrorMultiplicative} and concludes the proof.
\end{proof}

\section{Connection of $\Delta^\text{rel}(\Mu)$ with the $F$-distribution}\label{sec:f-distribution}

If we have $u(\Mu)^{T}\Sigma (u(\Mu)-\widetilde u(\Mu))=0$ we can give an even more precise description of the random variable $(\Delta^\text{rel}(\Mu)\|u(\Mu)\|_{\Sigma})^{2}/
\| u(\Mu)- \widetilde u(\Mu)\|_\Sigma^{2}$ and thus $\mathbb{P}\Big\{ w^{-1}  \Delta^\text{rel}(\Mu) \leq \frac{\|u(\Mu)-\widetilde u(\Mu) \|_\Sigma}{\|u(\Mu)\|_{\Sigma}} \leq w  \Delta^\text{rel}(\Mu) ~,~\forall \Mu \in \mathcal{P} \Big\}$: In fact, under this assumption the random variable $(\Delta^\text{rel}(\Mu)\|u(\Mu)\|_{\Sigma})^{2}/
\| u(\Mu)- \widetilde u(\Mu)\|_\Sigma^{2}$ follows an $F$-distribution with degrees of freedom $K$ and $K$; we also say $(\Delta^\text{rel}(\Mu)\|u(\Mu)\|_{\Sigma})^{2}/\| u(\Mu)- \widetilde u(\Mu)\|_\Sigma^{2} \sim F(K,K)$ as stated in the following lemma. 
\begin{lemma}[$F$-distribution]\label{lem:f-distribution}
Assume that there holds $u(\Mu)^{T}\Sigma (u(\Mu) - \urm)=0$. Then the random variable $(\Delta^\text{rel}(\Mu)\|u(\Mu)\|_{\Sigma})^{2}/\| u(\Mu)- \widetilde u(\Mu)\|_\Sigma^{2}$ follows an $F$-distribution with parameters $K$ and $K$.
\end{lemma}
\begin{proof}
We will exploit the fact that in order to show that a random variable $X$ follows an $F$-distribution it is sufficient to show that we have $X = \frac{Q_{1}/K_{1}}{Q_{2}/K_{2}}$, where $Q_{1}$ and $Q_{2}$ are independent random variables that follow a chi-square distribution with $K_{1}$ and $K_{2}$ degrees of freedom, respectively (see for instance \cite[Definition 9.7.1]{DeSc12}). Similar to \eqref{eq:chi-squared} we have that
\begin{equation*}
 \Delta^\text{rel}(\Mu)^{2} \frac{\|u(\Mu)\|_{\Sigma}^{2}}{\| u(\Mu)- \widetilde u(\Mu)\|_\Sigma^{2}} = \frac{  \frac{1}{K} \sum_{k=1}^{K} \Big(Z_{i}^{T} \big(u(\Mu)-\widetilde u(\Mu) \big) \Big)^{2} }{\frac{1}{K} \sum_{k=1}^{K} \Big(Z_{i}^{T} u(\Mu) \Big)^{2}}
\frac{\|u(\Mu)\|_{\Sigma}^{2}}{\| u(\Mu)- \widetilde u(\Mu)\|_\Sigma^{2}} = \frac{  \frac{1}{K} \sum_{k=1}^{K} \Big(Z_{i}^{T} \frac{\big(u(\Mu)-\widetilde u(\Mu) \big)}{\| u(\Mu)- \widetilde u(\Mu)\|_\Sigma}\Big)^{2} }{\frac{1}{K} \sum_{k=1}^{K} \Big(Z_{i}^{T} \frac{u(\Mu)}{\|u(\Mu)\|_{\Sigma}} \Big)^{2} } \sim \frac{\frac{1}{K}Q_{1}}{\frac{1}{K}Q_{2}},
\end{equation*}
where $Q_{1}$ and $Q_{2}$ follow a chi-squared distribution with $K$ degrees of freedom. It thus remains to show that under the above assumption $Q_{1}$ and $Q_{2}$ are independent. To that end, recall that $Z_{1},\hdots,Z_{K}$ are independent copies of $Z\sim \mathcal{N}(0,\Sigma)$. Per our assumption that $\Sigma$ is symmetric and positive semi-definite there exists a symmetric matrix $\Sigma^{1/2}$ that satisfies $\Sigma = \Sigma^{1/2}\Sigma^{1/2}$, exploiting for instance the Schur decomposition of $\Sigma$. Then we have $Z_{i} = \Sigma^{1/2}g_{i}$, where $g_{i}$ are $K$ independent copies of a standard Gaussian random vector. Thanks to the assumption $u(\Mu)^{T}\Sigma (u(\Mu) - \urm)=0$ the vectors $\hat{u}(\Mu):= \Sigma^{1/2}u(\Mu)$ and $\hat{e}(\Mu):=\Sigma^{1/2}(u(\Mu) - \urm)$ satisfy $\hat{u}(\Mu)^{T}\hat{e}(\Mu)=0$. As a consequence there exists an orthogonal matrix $V(\Mu) \in \R^{N\times N}$ that maps $\hat{u}(\Mu)$ to $e_{1}$ and $\hat{e}(\Mu)$ to $e_{2}$, respectively, where $e_{j}$, $j=1,\hdots,N$ denote the standard basis vectors in $\R^{N}$. We can then write for $i=1,\hdots,K$
\begin{equation}
(Z_{i}^{T}u(\Mu))^{2}=([\Sigma^{1/2}g_{i}]^{T} u(\Mu))^{2} = (g_{i}^{T}\Sigma^{1/2}u(\Mu))^{2} = (g_{i}^{T}\hat{u}(\Mu))^{2} = (g_{i}^{T}V(\Mu)^{T}V(\Mu)\hat{u}(\Mu))^{2} = ([V(\Mu)g_{i}]^{T}e_{1})^{2}.
\end{equation}
As there holds $V(\Mu)g_{i} \sim \mathcal{N}(0,I_{N})$, where $I_{N}$ is the identity matrix in $\R^{N}$, we have for  $\bar{g}_{i} \sim \mathcal{N}(0,I_{N})$
\begin{equation}
(Z_{i}^{T}u(\Mu))^{2} = (\bar{g}_{i}^{T}e_{1})^{2} = (\bar{g}^{1}_{i})^{2} \qquad \text{ and } \qquad  [Z_{i}^{T} \big(u(\Mu)-\widetilde u(\Mu) \big)]^{2} = (\bar{g}_{i}^{T}e_{2 })^{2} = (\bar{g}^{2}_{i})^{2}.
\end{equation}
Here, $\bar{g}^{s}_{i}$ denotes the $s$-th entry of the vector $\bar{g}_{i}$. As the entries of a standard Gaussian random vector are independent, we can thus conclude that the random variables $Q_{1}$ and $Q_{2}$ are independent as well. 
\end{proof}
We may then use the cumulative distribution function (cdf) of the $F$-distribution and an union bound argument to obtain the following result.
\begin{corollary}
Assume that there holds $u(\Mu)^{T}\Sigma (u(\Mu)-\widetilde u(\Mu))=0$. Then, we have 
\begin{equation}
  \mathbb{P}\Big\{ w^{-1} \Delta^\text{rel}(\Mu) \leq \frac{\|u(\Mu)-\widetilde u(\Mu) \|_\Sigma}{\|u(\Mu)\|_{\Sigma}} \leq w \Delta^\text{rel}(\Mu) ~,~\forall \Mu \in \mathcal{P} \Big\} \geq 1 - (\#\mathcal{P})\left(I_{\left(\frac{K}{K+w^{2}K}\right)}\left(\frac{K}{2},\frac{K}{2}\right) + 1 - I_{\left(\frac{Kw^{2}}{Kw^{2}+K}\right)}\left(\frac{K}{2},\frac{K}{2}\right)\right),
\end{equation}
where $I_{x}(a,b)$ is the regularized incomplete beta function defined as $(\int_{0}^{x} t^{a-1}(1-t)^{b-1} \,dt)/(\int_{0}^{1} t^{a-1}(1-t)^{b-1} \,dt)$. Note that the cdf $F(x,K,K)$ of an $F$-distribution with $K$ and $K$ degrees of freedom is given as $F(x,K,K)=I_{\left(\frac{Kx}{Kx+K}\right)}\left(\frac{K}{2},\frac{K}{2}\right)$.
\end{corollary}

\end{document}